\documentclass[11pt]{amsart}
\usepackage{amssymb,amsfonts,amsmath,latexsym,verbatim,amscd,amsthm}

\usepackage{hyperref}
\usepackage[colorinlistoftodos,prependcaption]{todonotes}

\usepackage{etoolbox}
\AtBeginEnvironment{proof}{\vspace{-14pt}}

\allowdisplaybreaks

\newtheorem{theorem}{Theorem}[section]
\newtheorem{proposition}{Proposition}[section]
\newtheorem{corollary}{Corollary}[section]
\newtheorem{lemma}{Lemma}[section]

\newtheorem{remark}{Remark}[section]
\newtheorem{claim}{Claim}

\newcommand{\R}{\mathbb R}

\newcommand{\mS}{\mathbb S}
\newcommand{\calH}{\mathcal H}

\newcommand{\calE}{\mathcal E}

\newcommand{\bfH}{\mathbf H}

\def\tr{\mathop{\rm tr}\nolimits}
\newcommand{\eps}{\epsilon}

\newcommand{\ra}{\rightarrow}

\def\uclhome{@ucl.ac.uk}
\def\jussieuhome{@imj-prg.fr}

\begin{document}

\title[Generic Uniqueness]{Generic uniqueness of expanders with vanishing relative entropy}
\author{Alix Deruelle}
\address{Alix Deruelle: 
Institut de Math\'ematiques de Jussieu, Paris Rive Gauche (IMJ-PRG) UPMC - Campus Jussieu, 4, place Jussieu Boite Courrier 247 - 75252 Paris Cedex 05, France}
\curraddr{}
\email{alix.deruelle\jussieuhome}

\author{Felix Schulze}
\address{Felix Schulze: 
  Department of Mathematics, University College London, 25 Gordon St,
  London WC1E 6BT, UK}
\curraddr{}
\email{f.schulze\uclhome}

\subjclass[2000]{}

\dedicatory{}


\begin{abstract}  
We define a relative entropy for two self-similarly expanding solutions to mean curvature flow of hypersurfaces, asymptotic to the same cone at infinity. Adapting work of White \cite{White87} and using recent results of Bernstein \cite{Bernstein17} and Bernstein-Wang \cite{BernsteinWang17}, we show that expanders with vanishing relative entropy are unique in a generic sense. This also implies that generically locally entropy minimising expanders are unique.  
\end{abstract}

\maketitle

\date{\today}

\section{Introduction}

We consider a hypersurface $\Gamma^{n-1} \subset \mS^{n}\subset \R^{n+1}$ of class $C^{5}$ and the cone
$$C(\Gamma) = \{tx\, |\, x \in \Gamma, t \in [0,\infty)\}\ .$$
In this paper, we focus on expanding solutions of the Mean Curvature Flow (MCF) coming out of $C(\Gamma)$, i.e. solutions that are invariant under parabolic rescalings starting at $C(\Gamma)$: the equation satisfied by an expander reflects the homogeneity of the initial condition $C(\Gamma)$ in a parabolic sense:
\begin{equation}
 \bfH = \frac{x^\perp}{2}\ .\label{equ-exp-MCF}
\end{equation}
By the work of Ilmanen \cite{Ilmanen:LecturesMCF}, see also \cite{Ding15}, it is known that there exists an expander $\Sigma$,
which is smooth away from a singular set of codimension 6, satisfying equation (\ref{equ-exp-MCF}) which is asymptotic to $C(\Gamma)$, i.e. its tangent cone at infinity is equal to $C(\Gamma)$. Moreover, $\Sigma$ can be chosen to be a local minimiser of the functional
\begin{equation}
\calE(M) = \int_M e^\frac{|x|^2}{4}\, d\calH^n.\label{ent-eqn-formal}
\end{equation}
Let us recall that from a variational viewpoint, expanders can be interpreted as critical points of the previous formal entropy (\ref{ent-eqn-formal}): the main issue is that this quantity associated to an expander is infinite.

This is in sharp contrast with shrinking solutions to the MCF since these solutions are critical points of the following well-defined entropy: 
\begin{eqnarray*}
\mathcal{F}(M):=\int_M e^{-\frac{|x|^2}{4}}\, d\calH^n<+\infty.
\end{eqnarray*}

To circumvent this issue, let $\Sigma_0$ be an expander coming out of a cone $C(\Gamma)$ and assume there exists another expander $\Sigma_1$ coming out of the same cone $C(\Gamma)$. Then, the relative entropy of $\Sigma_1$ and $\Sigma_0$ is formally defined by:
\begin{eqnarray}\label{eq:1}
\calE_{\Sigma_0,\Sigma_1}:=\lim_{R\rightarrow+\infty} \left(\int_{\Sigma_1 \cap B_R(0)} e^\frac{|x|^2}{4}\, d\calH^n - \int_{\Sigma_0 \cap B_R(0)} e^\frac{|x|^2}{4}\, d\calH^n\ \right).\label{rel-ent-intro-def}
\end{eqnarray}
A similar relative entropy, where $\Sigma_0$ is the asymptotic cone,  has been previously considered by Ilmanen, Neves and the second author for the network flow for regular networks to show uniqueness of expanders in their topological class, \cite{Ilm-Nev-Sch}. Note that in the case of networks each end of the asymptotic cone is a half-line, and thus an expander. 

In order to prove that the quantity (\ref{rel-ent-intro-def}) is well-defined, we need to establish a convergence rate for the exponential normal graph $u$ of $\Sigma_1$ over $\Sigma_0$ outside a sufficiently large compact set of the form 
\begin{eqnarray}\label{q:0e}
u=\textit{O}\left(r^{-n-1}e^{-r^2/4}\right).\label{dec-u-intro}
\end{eqnarray}
 This rate is sharp as shown by a unique continuation result at infinity proved by Bernstein \cite{Bernstein17}. Observe that estimating the difference of the two normal graphs of $\Sigma_i$ over the cone $C(\Gamma)$ given by comparing the expanders to their common initial condition $C(\Gamma)$ only yields {polynomial} decay: this is not sufficient to prove the expected decay (\ref{dec-u-intro}).

Again, we underline the fact that (\ref{rel-ent-intro-def}) is {defined} by taking differences rather than by considering a renormalization: this makes the analysis much harder since one has to match the asymptotics of such expanders in a much more precise way. Notice that a renormalization on an increasing sequence of exhausting balls in $\Sigma_0$ would have made the first variation of $\mathcal{E}$ vanish: see Theorem \ref{thm:1.5}. 

The main application of the existence of such a relative entropy is a generic uniqueness result for expanders. It can be stated roughly as follows:
\begin{theorem}[Generic uniqueness: {informal} statement]\label{rough-sta-gen-uni-intro}
The set of cones $C(\Gamma)$ that are smoothed out by more than one (suitable) expander with $0$ relative entropy is of first category in the Baire sense. 

In particular, the set of cones $C(\Gamma)$ that are smoothed out by more than one (suitable) locally entropy minimising expander is of first category in the Baire sense. 
\end{theorem}
We refer the reader to Theorem \ref{gen-uni-0-rel-ent} for a precise statement. Both the statement and the proof of Theorem \ref{rough-sta-gen-uni-intro} are motivated by the work of {White} on minimal surfaces: \cite[Section 7]{White87}. 

Before explaining the main steps of the proof of Theorem \ref{rough-sta-gen-uni-intro}, we would like to put our results into perspective with respect to other uniqueness and non-uniqueness results for self-similar solutions to MCF. Wang \cite{Wang14} has proved that two complete self-similarly shrinking solutions of MCF which are asymptotic to the same smooth cone at infinity must coincide. This result can be interpreted as a backward uniqueness result since shrinking solutions are ancient solutions to MCF. No such statement can be expected for (asymptotically conical) self-similarly expanding solutions to MCF without further assumptions. Indeed, Angenent, Chopp and Ilmanen \cite{Ang-Cho-Ilm} have shown that the asymptotic cone of the self-similarly shrinking four-handle saddle in $\R^3$ admits at least two different self-similarly expanding solutions to MCF coming out of it: see Ilmanen's notes on MCF \cite[Lecture 4]{Ilmanen:LecturesMCF} for further examples and more non-uniqueness results for other geometric flows. Finally, Bernstein \cite{Bernstein17} has identified the obstruction for two smooth asymptotically conical expanding solutions to MCF coming out of the same cone to coincide: the obstruction can be identified as the suitably rescaled limit at infinity of the difference of two such solutions, called the trace at infinity. We refer the reader to \cite{Bernstein17} for more details on this obstruction.   

Let us describe the main ideas that lead to the proof of Theorem \ref{rough-sta-gen-uni-intro}. In order to prove such a genericity statement, one needs to understand the (Banach) manifold structure of the moduli space of suitable expanders. It requires in particular to understand the Fredholm properties of the Jacobi operator associated to equation (\ref{equ-exp-MCF}): we rely on {work of Bernstein-Wang \cite{BernsteinWang17}, where they adapt White's approach to expanders of the Mean Curvature Flow.} The other ingredient is to make sense of the radial limit of the rescaled exponential normal graph $\hat{u}:=r^{n+1}e^{r^2/4}u$ in order to identify it with the differential of $\mathcal{E}$: see Corollary \ref{coro-def-trace-inf} and Theorem \ref{thm:1.5}. Following \cite{BernsteinWang17}, the radial limit of $\hat{u}$ is called the trace at infinity of $\hat{u}$ and it is denoted by $\tr_{\infty}^0(\hat{u})$.

We end this introduction by describing the structure of this paper:
Section \ref{sec-preliminaries} is essentially establishing technical preliminaries in order to define the relative entropy (\ref{rel-ent-intro-def}). The main result of the first part of this section is Theorem \ref{theo-ptwise-est}: sharp pointwise estimates on the exponential normal graph $u$ of an expander over another expander coming out of the same cone. The main tool is the maximum principle and it allows to estimate quantitatively the dependence of the multiplicative constants in front of the exponential weight $r^{-n-1}e^{-r^2/4}$. Similar crucial estimates on the first and second rescaled derivatives of $\hat{u}$ are obtained via Bernstein-Shi type estimates. The second part of Theorem \ref{theo-ptwise-est} establishes the corresponding statement for a Jacobi field associated to an asymptotically conical expander that vanishes at infinity. The first part of Section \ref{sec-preliminaries} ends with Corollary \ref{coro-def-trace-inf} that makes sense of the radial limit of the rescaled exponential normal graph $u$ mentioned above, called the trace of $\hat{u}$, and which has been introduced in \cite{Bernstein17}. We mention that Theorem \ref{theo-ptwise-est} is the pointwise version of the integral estimates proved by Bernstein \cite[Theorem 7.2]{Bernstein17} where a lower regularity on the cone can be assumed.

 The second part of Section \ref{sec-preliminaries} starts by analyzing the Taylor expansion  at infinity up to order $1$ of a Jacobi field associated to an expander:  the proof of Lemma \ref{lem: hessian estimate} uses the asymptotically conical geometry of such an expander in an essential way. Lemma \ref{lem: hessian estimate} is then used to estimate the difference of two Jacobi fields $v_i$, $i=0,1$ associated to two (a priori different) expanders $\Sigma_i$, $i=0,1$ coming out of a same cone: this is the content of Lemma \ref{lem: second jacobi estimate}. The length of its proof is due to the fact that one needs to linearize both the metric of $\Sigma_1$ over $\Sigma_0$  and the Jacobi field $v_1$ with respect to $v_0$.  

In Section \ref{ren-exp-ent-section}, we prove that the relative entropy (\ref{rel-ent-intro-def}) is well-defined: these are the contents of Proposition \ref{thm:1.3} together with Corollary \ref{thm:1.4}. Note that the estimate \eqref{q:0e} does not suffice directly to get the desired convergence in \eqref{eq:1}: it is necessary to exploit that the expander entropy only varies to second order around a critial point. Proving that the relative entropy is differentiable requires even more care to identify which non-zero terms show up at infinity: this is the purpose of Theorem \ref{thm:1.5}. As explained above, the differential of the relative entropy of two expanders coming out of a same cone can be identified with the trace at infinity of the rescaled exponential normal graph denoted by $\tr_{\infty}^0(\hat{u})$.

Section \ref{gen-uniqueness-section} comprises {of} an application of  Theorems \ref{theo-ptwise-est} and \ref{thm:1.5} together with the results of Bernstein-Wang \cite{BernsteinWang17} on the Fredholm properties of the Jacobi operator associated to an asymptotically conical expander to prove a generic uniqueness property formulated in Theorem \ref{rough-sta-gen-uni-intro}: again we refer to Theorem \ref{gen-uni-0-rel-ent} for a rigorous statement.
In Appendix \ref{app:NG}, we recall some facts about the geometry of normal graphs and Appendix \ref{app:interpolate} contains the statements of well-known interpolation inequalities.   \\ 

\textbf{Acknowledgements.}
The authors wish to thank Tom Ilmanen for sharing his ideas. {A.D. was supported by the grant ANR-17-CE40-0034 of the French National Research Agency ANR (project CCEM). F.S. was supported by a Leverhulme Trust Research Project Grant RPG-2016-174.}

\section{Preliminaries}\label{sec-preliminaries}

We will denote $r(x):=|x|$. Let $\Sigma$ be a self-expander, asymptotic to $C(\Gamma)$. It is known, see for example \cite{Ding15}, that outside a sufficiently large ball $B_{R_0}(0)$, where $R_0=R_0(\Gamma)$, we can write $\Sigma$ as a normal graph over $C(\Gamma)$, given by a function $f$. The asymptotic convergence yields via scaling the estimates
\begin{equation}\label{eq.-1}
 |\nabla^i_Cf(x)| \leq C r^{1-i}
\end{equation}
for $0\leq i \leq 5$. Using rotationally symmetric expanders as barriers, one obtains the following improved estimates, see \cite{Ding15}:
\begin{equation}\label{eq.0}
 |\nabla^i_Cf(x)| \leq C r^{-i-1}
\end{equation}
for $i=0,1,2$ and some $C \geq 0$.  Since $\Sigma_0$ is an expander, $2H=-x^{\perp}$ implies by differentiating along $\Sigma_0$ that the second fundamental form decays much faster than expected along the radial direction
\begin{equation}\label{eq.0.1}
\nabla^{\Sigma_0}H_{\Sigma_0}=\frac{1}{2}A_{\Sigma_0}(x^{\top},\cdot)\, .
\end{equation}
Assume that we have two expanders $\Sigma_0$ and $\Sigma_1$ asymptotic to $C(\Gamma)$.  For $R_0$ sufficiently large, we denote
$$\bar{E}_{i,{R_0}}= \Sigma_i\setminus  B_{R_0}(0)\, .$$
Using the above estimates we can write the end $\Sigma_{1,R_0} \subset \Sigma_1$ as an exponential normal graph over $\bar{E}_{0,{R_0}}$  with height function $u :  \bar{E}_{0,R_0} \ra \R$.  Note that in general we have $\bar{E}_{1,{R_0}} \neq \Sigma_{1,R_0}$ due to a small error in a collar region around $\mathbb{S}_{R_0}$, coming from the fact that the exponential normal coordinates over $\Sigma_0$ are not parallel to spheres centred at the origin.

\subsection{Pointwise estimates}
We first recall the equation satisfied by the height function $u$ on $\bar{E}_{0,{R_0}}$.
We have the following slight refinement of \cite[Lemma 8.2]{Bernstein17}, see also \cite[Lemma 5.2]{Ding15}:

\begin{lemma}\label{thm:1.1}
Let $\Sigma_0,\Sigma_1$ be two expanders asymptotic to the cone $C(\Gamma)$, and assume $ \bar{E}_{1,R_0} = \text{graph}_{\bar{E}_{0,R_0}}(u)$ with $u : \bar{E}_{0,R_0} \ra \R$. Then $u$ satisfies 
\begin{equation}\label{eq.1}
 \Delta_{\Sigma_0}u + \frac{1}{2}\langle x, \nabla^{\Sigma_0} u\rangle+ \Big( |A_{\Sigma_0}|^2 - \frac{1}{2}\Big) u = Q(u,\nabla^{{\Sigma_0}}u, \nabla^{2}_{{\Sigma_0}}u)\, ,
\end{equation}
where $Q(u,\nabla^{{\Sigma_0}}u, \nabla^{2}_{{\Sigma_0}}u) = \sum_{i+j =2} \nabla^{i}_{{\Sigma_0}}u * \nabla^{j}_{{\Sigma_0}}u$ and we have used $*$-notation to suppress any smooth coefficients, which, together with their derivatives, are uniformly bounded as $r \rightarrow \infty$. In particular, by estimating more carefully we have
\begin{equation}\label{eq.1.1}
\big|Q(u,\nabla^{\Sigma_0}u, \nabla^{2}_{\Sigma_0}u)\big| \leq C \big( r^{-2} |u| + r^{-3} |\nabla_{\Sigma_0} u| \big)\, ,
\end{equation}
where $C$ depends continuously on the the link $\Gamma$.\\
\end{lemma}
\begin{proof} The estimates from writing both $\Sigma_0,\Sigma_1$ as a graph over $C(\Gamma)$ carry directly over to $u$, so we obtain from \eqref{eq.-1} that
\begin{equation}\label{eq:app.1}
 |\nabla^i_{\Sigma_0}u| \leq C r^{1-i}
\end{equation}
for all $ 0\leq i\leq 5$. From \eqref{eq.0} we obtain the improved estimates
\begin{equation}\label{eq:app.2}
 |\nabla^i_{\Sigma_0}u| \leq C r^{-i-1}
\end{equation}
for $i=0,1,2$. The estimates \eqref{eq.1} and \eqref{eq.1.1} then follow from \eqref{eq:ng.2} and \eqref{eq:ng.5}.   Note that \eqref{eq.1.1} does not directly follow from the quadratic structure of $Q$ together with \eqref{eq:app.2}, but from a direct estimation of the error terms in \eqref{eq:ng.2} and \eqref{eq:ng.5}, using the faster decay of the second fundamental form in the radial direction, \eqref{eq.0.1}.
\end{proof}

The next theorem ensures the integral estimates due to Bernstein \cite{Bernstein17} hold in a pointwise sense up to second order provided the (convergence to the) cone at infinity is sufficiently smooth.

\begin{theorem}\label{theo-ptwise-est}
 Let $(\Sigma_i)_{i=0,1}$ be two expanders asymptotic to the same cone $C(\Gamma)$ such that $\nabla_{\Sigma_i}^jA_i=\textit{O}(r^{-1-j})$ for $j\in\{0,\ldots, 3\}$ and $i=0,1$. Furthermore, let $\Sigma_{1,R_0} = \text{graph}_{\bar{E}_{0,R_0}}(u)$ with $u : \bar{E}_{0,R_0} \ra \R$. 
 \begin{enumerate}
 \item Then $u$ satisfies pointwise,
\begin{equation}\label{ptwise-est-u}
\begin{split}
\sup_{\bar{E}_{0,2R_0}}r^i\left| \nabla^{i}_{\Sigma_0}\left(r^{n+1}e^{\frac{r^2}{4}}u\right)\right|\leq C,
 \end{split}
  \end{equation}
  for $i=0,1,2$ and where $C$ depends only on $n$, $\sup_{\bar{E}_{0,R_0}}r^{1+j}|\nabla^{\Sigma_i,j}A_{\Sigma_i}|$ for $j\in\{0,\ldots, 3\}$, $i=0,1$ and $\sup_{\bar{E}_{0,R_0}}|u|$.\\
  
\item Assume $v: \bar{E}_{0,R_0} \ra \R$ is an approximate smooth Jacobi field that vanishes at infinity in the following sense:
\begin{equation}
\begin{split}
\Delta_{\Sigma_0}v+\frac{1}{2}\langle x,\nabla^{\Sigma_0}v\rangle - \frac{v}{2} &=V_1v+ Q,\label{app-jac-fiel}\\
\lim_{r\rightarrow +\infty}v&=0,
\end{split}
\end{equation}
where $V_1$ is a function on $\bar{E}_{0,R_0}$ such that $V_1=\textit{O}(r^{-2})$, and where $Q$ is a given smooth function on $\bar{E}_{0,R_0}.$ 
\begin{itemize}
\item If $Q=\textit{O}\left(r^{-n-3}e^{-\frac{r^2}{4}}\right)$ then $v=\textit{O}\left(r^{-n-1}e^{-\frac{r^2}{4}}\right)$.\\
\item If $Q=\textit{O}\left(r^{-n-1}e^{-\frac{r^2}{4}}\right)$ then $v=\textit{O}\left(r^{-n-1+\varepsilon}e^{-\frac{r^2}{4}}\right)$ for every positive $\varepsilon$.
\end{itemize}
\end{enumerate}
\end{theorem}
\begin{remark}
We decided to state the assumptions (\ref{app-jac-fiel}) satisfied by $v$ in this non-sharp form since it is sufficient for our purpose. The sharp version would ask $\lim_{r\rightarrow +\infty}r^{-1}v=0$.
\end{remark}
\begin{proof}
Let us start to prove the $C^0$ estimates on the behavior at infinity of $u$ and $v$ (depending on the behavior at infinity of the data $Q$). Since the proof will be similar, we essentially give a proof for the decay of $u$. Recall that $u$ satisfies schematically
\begin{equation}\label{equ-sch-evo-u}
\Delta_{\Sigma_0}u+\frac{1}{2}\langle x,\nabla^{\Sigma_0}u\rangle - \frac{u}{2}=V_1u+\langle V_2,\nabla^{\Sigma_0}u\rangle,
\end{equation}
where $V_1$ is a function on $\bar{E}_{0,R_0}$ such that $V_1=\textit{O}(r^{-2})$ and where $V_2$ is a vector field on $\bar{E}_{0,R_0}$ such that $V_2=\textit{O}(r^{-3})$ (by Lemma \ref{thm:1.1}). 
Then a tedious but straightforward computation using the conical geometry at infinity of the expanding soliton $\Sigma_0$ shows that if $w:=r^{-n-1}e^{-\frac{r^2}{4}}$, then
\begin{eqnarray*}
\Delta_{\Sigma_0}w+\frac{1}{2}\langle x,\nabla^{\Sigma_0}w\rangle - \frac{w}{2}=\textit{O}(r^{-2})\, w.
\end{eqnarray*}
In particular, 
\begin{eqnarray*}
\Delta_{\Sigma_0}w+\frac{1}{2}\langle x,\nabla^{\Sigma_0}w\rangle - \frac{w}{2}-V_1w-\langle V_2,\nabla^{\Sigma_0}w\rangle=\textit{O}(r^{-2})w\, .
\end{eqnarray*}
 For $A$ a positive constant, define $w_A:=e^{Ar^{-2}}w=r^{-n-1}e^{Ar^{-2}-\frac{r^2}{4}}$ and note that
\begin{equation}
\begin{split}
\Delta_{\Sigma_0}w_A+\frac{1}{2}\langle x,\nabla^{\Sigma_0}w_A\rangle - \frac{w_A}{2}&\leq\\
& V_1w_A+\langle V_2,\nabla^{\Sigma_0}w_A\rangle-\frac{A}{2}r^{-2}w_A,\label{est-barrier-fct}
\end{split}
\end{equation}
if $r\geq R_1=R_1(A)$ is large enough.

Now, if $B$ is a positive constant to be chosen later, one gets:
\begin{eqnarray*}
\Delta_{\Sigma_0}(u-Bw_A)+\frac{1}{2}\langle x,\nabla^{\Sigma_0}(u-Bw_A)\rangle -\langle V_2,\nabla^{\Sigma_0}(u-Bw_A)\rangle\\
\geq  \left(\frac{1}{2}+V_1\right)(u-Bw_A).
\end{eqnarray*}

The constant $A$ being fixed, take $R_1$ such that $1/2+V_1\geq 1/4$ on $\{r\geq R_1\}$ and choose $B$ large enough such that $$\sup_{r=R_1}(u-Bw_A)\leq 0.$$ By applying the maximum principle to the previous differential inequality satisfied by $u-Bw_A$, we arrive at
\begin{eqnarray}
\sup_{R_1\leq r\leq R}(u-Bw_A)=\max\left\{0,\sup_{r=R}(u-Bw_A)\right\}.
\end{eqnarray}
Since both $u$ and $w_A$ go to $0$ at infinity, the expected decay on $u$ follows and $B$ depends only on the expected quantities.

The same idea applies to the estimate of a solution $v$ satisfying (\ref{app-jac-fiel}) with 
$$Q=\textit{O}\left(r^{-n-3}e^{-\frac{r^2}{4}}\right)\, .$$
Indeed, according to (\ref{est-barrier-fct}), the function $v-Bw_A$ satisfies:
\begin{equation*}
\begin{split}
&\Delta_{\Sigma_0}(v-Bw_A)+\frac{1}{2}\langle x,\nabla^{\Sigma_0}(v-Bw_A)\rangle - \frac{(v-Bw_A)}{2}\geq\\
& \frac{BA}{2r^2}w_A-\frac{C}{r^2}r^{-n-1}e^{-\frac{r^2}{4}}\geq\frac{BA-2C}{2r^2}w_A>0,
\end{split}
\end{equation*}
if $B>B(A,C)$ and  
where we used the assumption on the righthand side of (\ref{app-jac-fiel}). Again, by choosing $B$ sufficiently large such that $\sup_{r=R_1}(v-Bw_A)\leq 0$, the maximum principle applied to the previous differential inequality shows that $\sup_{R_1\leq r\leq R}(v-Bw_A)=\max\{0,\sup_{r=R}(v-Bw_A)\}$. Since both $v$ and $w_A$ go to $0$ at infinity, one gets the expected decay on $v$.

In case 
$$Q=\textit{O}\left(r^{-n-1}e^{-\frac{r^2}{4}}\right)$$ 
we have that the function $w^{\varepsilon}:=r^{-n-1+\varepsilon}e^{-\frac{r^2}{4}}$ for some positive $\varepsilon$ is a good barrier function by the following estimates:
\begin{equation}
\begin{split}
&\Delta_{\Sigma_0}w^{\varepsilon}+\frac{1}{2}\langle x,\nabla^{\Sigma_0}w^{\varepsilon}\rangle - \frac{1-\varepsilon}{2}w^{\varepsilon}=\textit{O}(r^{-2})w^{\varepsilon},\\
&\Delta_{\Sigma_0}\left(v-Bw^{\varepsilon}\right)+\frac{1}{2}\langle x,\nabla^{\Sigma_0}\left(v-Bw^{\varepsilon}\right)\rangle =\\
& \left(\frac{1}{2}+\textit{O}(r^{-2})\right)\left(v-Bw^{\varepsilon}\right)+B\left(\frac{\varepsilon}{2}+\textit{O}(r^{-2})\right)w^{\varepsilon}+\textit{O}\left(r^{-n-1}e^{-\frac{r^2}{4}}\right),
\end{split}
\end{equation}
where $B$ is any positive constant. Since $w^{\varepsilon}$ decays slower than the data $Q$, the expected estimate can be proved along the same lines of the previous cases.

Finally, it remains to prove estimates on the first and second derivatives of the rescaled function $\hat{u}:=r^{n+1}e^{\frac{r^2}{4}}u.$ Before doing so, we note that the interpolation inequalities (\ref{lemm:interpolation}) applied to $u$ show that the higher derivatives of $u$ decay exponentially, i.e. 
 $$|\nabla_{\Sigma_0}^iu|\leq Ce^{-\frac{\varepsilon r^2}{4}}$$ 
for $i\in\{1,\dots ,4\}$ and some $\varepsilon\in(0,1)$. Note that the constant again depends continuously on 
\begin{eqnarray*}
\text{$n$, $\sup_{\bar{E}_{0,R_0}}r^{1+j}|\nabla^{\Sigma_i,j}A_{\Sigma_i}|$ for $j\in \{0,\ldots,3\}$, $i=0,1$ and $\sup_{\bar{E}_{0,R_0}}|u|$.}
\end{eqnarray*}
 and thus continuously on the link $\Gamma$.

By the previous discussion, we could use equation (\ref{eq.1}) by treating the righthand side as a data since it decays much faster than the expected decay for $u$. Nonetheless, since Bernstein-Shi type estimates use the linear structure of the equation under consideration in an essential way, we work with (\ref{equ-sch-evo-u}) instead: the function $V_1$ (respectively the vector field $V_2$) is now decaying like $r^{-2}$ (respectively like $r^{-3}$) together with its first and second covariant derivatives: 
\begin{equation*}
\nabla_{\Sigma_0}^iV_1=\textit{O}(r^{-2-i}),\quad \nabla_{\Sigma_0}^iV_2=\textit{O}(r^{-3-i}),\quad i=0,1,2,
\end{equation*}
where the estimates $\textit{O}(\cdot)$ are continuously depending on 
\begin{eqnarray*}
\text{$n$, $\sup_{\bar{E}_{0,R_0}}r^{1+j}|\nabla^{\Sigma_i,j}A_{\Sigma_i}|$ for $j\in \{0,\ldots,3\}$, $i=0,1$ and $\sup_{\bar{E}_{0,R_0}}|u|$.}
\end{eqnarray*}
From this remark, we compute the equation satisfied by $\hat{u}$ (see also \cite[Section 7]{Bernstein17}):
\begin{eqnarray}\label{evo-equ-u-hat}
\Delta_{\Sigma_0}\hat{u}-\frac{1}{2}\langle x,\nabla^{\Sigma_0}\hat{u}\rangle&=& W_1\hat{u}+\langle W_2,\nabla^{\Sigma_0} \hat{u}\rangle,
\end{eqnarray}
where $W_1$ is a function defined on $\bar{E}_{0,R_0}$ such that $\nabla^{\Sigma_0,i}W_1=\textit{O}(r^{-2-i})$ for $i=0,1,2$ and where $W_2$ is a vector field on $\bar{E}_{0,R_0}$ such that $\nabla^{\Sigma_0,i}W_2=\textit{O}(r^{-1-i})$ for $i=0,1,2$. Again, notice that all the estimates $\textit{O}(\cdot)$ in (\ref{evo-equ-u-hat}) are depending on $u$ and the rescaled derivatives of the second fundamental form (up to three) only.

The second step consists in computing the evolution equation satisfied by the gradient of $\hat{u}$:
\begin{equation*}
\begin{split}
\left(\Delta_{\Sigma_0}-\frac{1}{2}\langle x,\nabla^{\Sigma_0}\cdot\rangle-\frac{1}{2}\right)\left(\nabla^{\Sigma_0}\hat{u}\right)&= \hat{u}\nabla^{\Sigma_0}W_1+W_1\nabla^{\Sigma_0}\hat{u}\\
&\ \ \ + \nabla^{\Sigma_0}W_2\ast \nabla^{\Sigma_0} \hat{u}+W_2\ast\nabla^{\Sigma_0,2}\hat{u}.
\end{split}
\end{equation*}
In particular, by considering the pointwise squared norm $| \nabla^{\Sigma_0}\hat{u}|^2$:
\begin{equation}\label{est-evo-equ-norm-first-der-u-hat}
\begin{split}
\left(\Delta_{\Sigma_0}-\frac{1}{2}\langle x,\nabla^{\Sigma_0}\cdot\rangle\right)| \nabla^{\Sigma_0}\hat{u}|^2&=2|\nabla^{\Sigma_0,2}\hat{u}|^2+|\nabla^{\Sigma_0}\hat{u}|^2+2\hat{u}\langle\nabla^{\Sigma_0}W_1,\nabla^{\Sigma_0}\hat{u}\rangle\\
&\ \ \ +2W_1|\nabla^{\Sigma_0}\hat{u}|^2+2\langle\nabla^{\Sigma_0}W_2\ast \nabla^{\Sigma_0} \hat{u},\nabla^{\Sigma_0}\hat{u}\rangle\\
&\ \ \ +\langle W_2\ast\nabla^{\Sigma_0,2}\hat{u},\nabla^{\Sigma_0}\hat{u}\rangle\\
&\geq |\nabla^{\Sigma_0,2}\hat{u}|^2+|\nabla^{\Sigma_0}\hat{u}|^2-c_1r^{-2}|\nabla^{\Sigma_0}\hat{u}|^2\\
&\ \ \ -c_2r^{-4}\hat{u}^2,
\end{split}
\end{equation}
where $c_1$ and $c_2$ are positive constants independent of $r\geq R_0$ and where we used Young's inequality to absorb the second derivatives of $\hat{u}$.

Finally we consider the norm of the rescaled derivatives $r^2|\nabla^{\Sigma_0}\hat{u}|^2$ of $\hat{u}$. 

Recall first that $r^2$ is an approximate eigenfunction of the drift laplacian $\Delta_{\Sigma_0}-\frac{1}{2}\langle x,\nabla^{\Sigma_0}\cdot\rangle$ associated to the eigenvalue $-1$:
\begin{eqnarray}
\Delta_{\Sigma_0}r^2-\frac{1}{2}\langle x,\nabla^{\Sigma_0}r^2\rangle=-r^2+\textit{O}(1).\label{evo-equ-pot-fct}
\end{eqnarray}
Therefore, (\ref{est-evo-equ-norm-first-der-u-hat}) together with the previous observation lead to:
\begin{equation}\label{evo-equ-resc-der-u-hat}
\begin{split}
\left(\Delta_{\Sigma_0}-\frac{1}{2}\langle x,\nabla^{\Sigma_0}\cdot\rangle\right)&\left(r^2|\nabla^{\Sigma_0}\hat{u}|^2\right)=r^2\left(\Delta_{\Sigma_0}-\frac{1}{2}\langle x,\nabla^{\Sigma_0}\cdot\rangle\right)|\nabla^{\Sigma_0}\hat{u}|^2\\
&\ \ \ +2\langle\nabla^{\Sigma_0}r^2,\nabla^{\Sigma_0}|\nabla^{\Sigma_0}\hat{u}|^2\rangle\\
&\ \ \ +|\nabla^{\Sigma_0}\hat{u}|^2
\left(\Delta_{\Sigma_0}-\frac{1}{2}\langle x,\nabla^{\Sigma_0}\cdot\rangle\right)r^2\\
&\geq \frac{r^2}{2}|\nabla^{\Sigma_0,2}\hat{u}|^2-c_1r^{-2}\left(r^2|\nabla^{\Sigma_0}\hat{u}|^2\right)-c_2r^{-2}\hat{u}^2.
\end{split}
\end{equation}
Moreover, (\ref{evo-equ-u-hat}) implies the following differential inequality:
\begin{eqnarray}\label{evo-equ-u-hat-2}
\left(\Delta_{\Sigma_0}-\frac{1}{2}\langle x,\nabla^{\Sigma_0}\cdot\rangle\right)\hat{u}^2&\geq&|\nabla^{\Sigma_0}\hat{u}|^2-c_1r^{-2}|\hat{u}|^2.
\end{eqnarray}
 We can start to prove so called Bernstein-Shi type estimates by considering the function $F:=(a^2+\hat{u}^2)r^2|\nabla^{\Sigma_0}\hat{u}|^2$ where $a$ is a positive constant to be chosen later. From now on, we denote by $c$ a positive constant that is independent of the radial function $r\geq R_0$ and which may vary from line to line. Inequalities (\ref{evo-equ-resc-der-u-hat}) and (\ref{evo-equ-u-hat-2}) give:
 \begin{equation}
 \begin{split}
\Big(&\Delta_{\Sigma_0}-\frac{1}{2}\langle x,\nabla^{\Sigma_0}\cdot\rangle\Big)F=r^2|\nabla^{\Sigma_0}\hat{u}|^2\left(\Delta_{\Sigma_0}-\frac{1}{2}\langle x,\nabla^{\Sigma_0}\cdot\rangle\right)\hat{u}^2\\
&\qquad\qquad+2\langle\nabla^{\Sigma_0}\hat{u}^2,\nabla^{\Sigma_0}(r^2|\nabla^{\Sigma_0}\hat{u}|^2)\rangle\\
&\qquad\qquad+(a^2+\hat{u}^2)\left(\Delta_{\Sigma_0}-\frac{1}{2}\langle x,\nabla^{\Sigma_0}\cdot\rangle\right)(r^2|\nabla^{\Sigma_0}\hat{u}|^2)\\
&\geq \frac{\left(r^2|\nabla^{\Sigma_0}\hat{u}|^2\right)^2}{r^2}-\frac{c}{r^2}\hat{u}^2\left(r^2|\nabla^{\Sigma_0}\hat{u}|^2\right)+2\langle\nabla^{\Sigma_0}\hat{u}^2,\nabla^{\Sigma_0}(r^2|\nabla^{\Sigma_0}\hat{u}|^2)\rangle\\
&\qquad\qquad+(a^2+\hat{u}^2)\left(\frac{r^2}{2}|\nabla^{\Sigma_0,2}\hat{u}|^2-cr^{-2}\left(r^2|\nabla^{\Sigma_0}\hat{u}|^2\right)-cr^{-2}\hat{u}^2\right)\\
&\geq \frac{\left(r^2|\nabla^{\Sigma_0}\hat{u}|^2\right)^2}{2r^2}+(a^2+\hat{u}^2)\frac{r^2}{2}|\nabla^{\Sigma_0,2}\hat{u}|^2-\frac{c}{r^2}F-\frac{c}{r^2}\hat{u}^2(\hat{u}^2+a^2)\\
&\qquad\qquad+2\langle\nabla^{\Sigma_0}\hat{u}^2,\nabla^{\Sigma_0}(r^2|\nabla^{\Sigma_0}\hat{u}|^2)\rangle.
\end{split}
\end{equation}
Now,
\begin{equation}
\begin{split}
2|\langle\nabla^{\Sigma_0}\hat{u}^2,\nabla^{\Sigma_0}(r^2|\nabla^{\Sigma_0}\hat{u}|^2)\rangle|&\leq c|\hat{u}|(r^2|\nabla^{\Sigma_0}\hat{u}|^2)|\nabla^{\Sigma_0,2}\hat{u}|\\
&\ \ \ +c|\hat{u}||\nabla^{\Sigma_0}\hat{u}|^2(r|\nabla^{\Sigma_0}\hat{u}|)\\
&\leq \frac{\left(r^2|\nabla^{\Sigma_0}\hat{u}|^2\right)^2}{4r^2}+c\hat{u}^2r^2|\nabla^{\Sigma_0,2}\hat{u}|^2+c\frac{F}{r^2}.
\end{split}
\end{equation}
Consequently, if $a$ is taken proportional to $\sup_{\bar{E}_{0,R_0}}|\hat{u}|$, i.e. if $$a:=\alpha \sup_{\bar{E}_{0,R_0}}|\hat{u}|,$$ with $\alpha$ a universal positive constant sufficiently large such that the terms involving the second derivatives of $\hat{u}$ can be absorbed, then
 \begin{equation}\label{evo-equ-F}
\begin{split}
r^2 \left(\Delta_{\Sigma_0}-\frac{1}{2}\langle x,\nabla^{\Sigma_0}\cdot\rangle\right)F&\geq c\left(r^2|\nabla^{\Sigma_0}\hat{u}|^2\right)^2-cF-c\hat{u}^2(\hat{u}^2+a^2)\\
&\geq \frac{c}{a^4}F^2 -ca^4.
\end{split}
\end{equation}
In order to use the maximum principle, we need to localize the previous differential inequality with the help of a cut-off function $\varphi_R:\bar{E}_{0,R_0}\rightarrow[0,1]$ such that $\varphi_R(x)=\eta(r(x)/R)$ where $\eta:[0,+\infty)\rightarrow[0,1]$ is a non-increasing smooth function with compact support in $[0,2]$ such that $\eta\equiv 1$ on $[0,1]$ and $(\eta')^2\leq c\eta$ together with $\eta''\geq -c$ for some positive constant $c$. Notice by construction that:
\begin{equation}
r^2\Delta_{\Sigma_0}\varphi_R\geq -c,\quad \frac{r^2|\nabla^{\Sigma_0}\varphi_R|^2}{\varphi_R}\leq c,\quad -\langle x,\nabla^{\Sigma_0}\varphi_R\rangle\geq 0.
\end{equation}
Therefore, by discarding the drift term $-\langle x,\nabla^{\Sigma_0}\varphi_R\rangle$,
 \begin{equation}\label{evo-equ-F-local}
\begin{split}
r^2 \varphi_R&\Big(\Delta_{\Sigma_0}-\frac{1}{2}\langle x,\nabla^{\Sigma_0}\cdot\rangle\Big)(\varphi_R F)\\
&\geq r^2 (\varphi_RF)\Big(\Delta_{\Sigma_0}-\frac{1}{2}\langle x,\nabla^{\Sigma_0}\cdot\rangle\Big)\varphi_R\\
&\ \ \ +2r^2\langle\nabla^{\Sigma_0}\varphi_R,\varphi_R \nabla^{\Sigma_0}F\rangle+r^2\varphi_R^2\Big(\Delta_{\Sigma_0}-\frac{1}{2}\langle x,\nabla^{\Sigma_0}\cdot\rangle\Big) F\\
&\geq -c(\varphi_R F)+2r^2\langle \nabla^{\Sigma_0}\varphi_R,\nabla^{\Sigma_0}(\varphi_R F)\rangle-2r^2\frac{|\nabla^{\Sigma_0}\varphi_R|^2}{\varphi_R}(\varphi_RF)\\
&\ \ \ +\frac{c}{a^4}(\varphi_R F)^2 -ca^4\\
&\geq 2r^2\langle \nabla^{\Sigma_0}\varphi_R,\nabla^{\Sigma_0}(\varphi_R F)\rangle+ \frac{c}{a^4}(\varphi_R F)^2 -ca^4,
\end{split}
\end{equation}
which gives the expected result by applying the maximum principle to $\varphi_R F$: 
\begin{equation}
\sup_{\bar{E}_{0,R_0}}r^2|\nabla^{\Sigma_0}\hat{u}|^2\leq c\bigg(\sup_{\bar{E}_{0,R_0}}|\hat{u}|^2+\sup_{\partial\bar{E}_{0,R_0}}r^2|\nabla^{\Sigma_0}\hat{u}|^2\bigg).
\end{equation}
By invoking local parabolic (or elliptic) estimates for $\hat{u}$, $$\sup_{\bar{E}_{0,R_0}}r^2|\nabla^{\Sigma_0}\hat{u}|^2\leq c\sup_{\bar{E}_{0,R_0/2}}|\hat{u}|^2.$$

In order to get an estimate on the second covariant derivatives of the rescaled function $\hat{u}$, we proceed similarly to the previous Bernstein-Shi type estimates on the first covariant derivatives of $\hat{u}$. Starting from (\ref{evo-equ-u-hat}), the tensor $\nabla^{\Sigma_0,2}\hat{u}$ satisfies the following qualitative evolution equation:
\begin{equation*}
\begin{split}
\left(\Delta_{\Sigma_0}-\frac{1}{2}\langle x,\nabla^{\Sigma_0}\cdot\rangle-1\right)\left(\nabla^{\Sigma_0,2}\hat{u}\right)&=\hat{u}\nabla^{\Sigma_0,2}W_1+\nabla^{\Sigma_0}\hat{u}\ast\nabla^{\Sigma_0}W_1\\
&\ \ \ +W_1\nabla^{\Sigma_0,2}\hat{u}+ \nabla^{\Sigma_0,2}W_2\ast \nabla^{\Sigma_0} \hat{u}\\
&\ \ \ +\nabla^{\Sigma_0}W_2\ast \nabla^{\Sigma_0,2} \hat{u}+W_2\ast\nabla^{\Sigma_0,3}\hat{u}.
\end{split}
\end{equation*}
In particular, by considering the pointwise squared norm $| \nabla^{\Sigma_0,2}\hat{u}|^2$ together with the bounds on (the covariant derivatives) of $W_1$ and $W_2$:
\begin{equation}\label{est-evo-equ-norm-sec-der-u-hat}
\begin{split}
\left(\Delta_{\Sigma_0}-\frac{1}{2}\langle x,\nabla^{\Sigma_0}\cdot\rangle\right)| \nabla^{\Sigma_0,2}\hat{u}|^2&\geq |\nabla^{\Sigma_0,2}\hat{u}|^2+|\nabla^{\Sigma_0,2}\hat{u}|^2\\
&\ \ \ -c_1r^{-2}|\nabla^{\Sigma_0,2}\hat{u}|^2-c_2r^{-4}|\nabla^{\Sigma_0}\hat{u}|^2\\ 
&\ \ \ -c_3r^{-6}\hat{u}^2,
\end{split}
\end{equation}
where $c_1$ and $c_2$ are positive constants independent of $r\geq R_0$ and where we used Young's inequality to absorb the third derivatives of $\hat{u}$. Using again (\ref{evo-equ-pot-fct}), one can absorb the linear term $|\nabla^{\Sigma_0,2}\hat{u}|^2$ with the help of the function $r^4$ to get:
\begin{equation}\label{evo-equ-resc-sec-der-u-hat}
\begin{split}
\left(\Delta_{\Sigma_0}-\frac{1}{2}\langle x,\nabla^{\Sigma_0}\cdot\rangle\right)\left(r^4| \nabla^{\Sigma_0,2}\hat{u}|^2\right)&\geq \frac{1}{2}r^4|\nabla^{\Sigma_0,3}\hat{u}|^2\\
&\ \ \ -c_1r^{-2}\left(r^4|\nabla^{\Sigma_0,2}\hat{u}|^2\right)\\
&\ \ \ -c_2|\nabla^{\Sigma_0}\hat{u}|^2-c_3r^{-2}\hat{u}^2.
\end{split}
\end{equation}
From there, one considers similarly an auxiliary function $$F_2:=(a^2+r^2|\nabla^{\Sigma_0}\hat{u}|^2)r^4|\nabla^{\Sigma_0,2}\hat{u}|^2,$$ where $a$ is a positive constant to be defined later. One can show that $F_2$ satisfies an analogous differential inequality to the one satisfied by $F$ given by (\ref{evo-equ-F}) by using the previous bound on $\nabla^{\Sigma_0}\hat{u}$: in this case, (\ref{evo-equ-resc-der-u-hat}) plays the role of (\ref{evo-equ-u-hat-2}) and the constant $a$ is again chosen proportionally to $\sup_{\bar{E}_{0,R_0}}|\hat{u}|$. Localizing the auxiliary function $F_2$ in order to use the maximum principle works exactly the same way we did in (\ref{evo-equ-F-local}).
\end{proof}

As a consequence of Theorem \ref{theo-ptwise-est}, one can make sense of the radial limit of $\hat{u}$, called the $0$-trace of $\hat{u}$ at infinity and denoted by $\tr_{\infty}^0\hat{u}$ (see \cite{BernsteinWang17}):
\begin{corollary}\label{coro-def-trace-inf}
Under the assumptions of Theorem \ref{theo-ptwise-est}, the radial limit 
\begin{eqnarray}\label{eq:convergence}
\tr_{\infty}^0(\hat{u}):=\lim_{r\rightarrow+\infty}r^{n+1}e^{\frac{r^2}{4}}u,
\end{eqnarray}
 exists and defines a  $C^{1,1}$ function on the link $\Gamma$ of the asymptotic cone $C(\Gamma)$.
\end{corollary}
\begin{remark}
 The $0$-trace at infinity of $\hat{u}$ is shown to be a function in $H^1(\Gamma)$ in \cite{BernsteinWang17}, with the corresponding convergence to \eqref{eq:convergence} only in $L^2$. The convergence in \eqref{eq:convergence} in Corollary \ref{coro-def-trace-inf} is in $C^{1,\alpha}$  by \eqref{ptwise-est-u} and provides the continuity of the differential of $\tr^0_{\infty}(\hat{u})$ as a function on the link $\Gamma$ by assuming more regularity on the asymptotic cone at infinity than in \cite{BernsteinWang17}. Assuming both the asymptotic cone and the convergence to it to be smooth, one can show $\tr^0_{\infty}(\hat{u})$ is a smooth function on $\Gamma$ by adapting the proof of Corollary \ref{coro-def-trace-inf}.
\end{remark}
\begin{proof}
According to equation (\ref{evo-equ-u-hat}) together with Theorem \ref{theo-ptwise-est}, especially (\ref{ptwise-est-u}),  one arrives at:
\begin{eqnarray*}
\langle x,\nabla^{\Sigma_0}\hat{u}\rangle=\textit{O}(r^{-2}).
\end{eqnarray*}
Integrating this differential relation along the radial direction proves the existence of a limit denoted by $\tr_{\infty}^0(\hat{u})$ as the radial distance goes to $+\infty$. Moreover, it also shows that $\hat{u}$ converges to $\tr_{\infty}^0(\hat{u})$ quadratically. Finally, since the rescaled derivatives of $\hat{u}$ are bounded by (\ref{ptwise-est-u}), an application of Arzel\`a-Ascoli's theorem leads to the expected regularity of $\tr^0_{\infty}(\hat{u})$, i.e. $\tr^0_{\infty}(\hat{u})\in C^1(\Gamma)$.
 \end{proof}

\subsection{Estimates for Jacobi fields}
Let  $(\Gamma_s)_{-\eps<s<\eps}$ be a continuously differentiable  family of $C^{5}$ hypersurfaces of $\mathbb{S}^n$. Assume that $(\Sigma_{0,s}, \Sigma_{1,s})_{-\eps<s<\eps}$ is a continuously differentiable family of expanders such that both $\Sigma_{0,s}, \Sigma_{1,s}$ are asymptotic to $C_s:=C(\Gamma_s)$.
  Let $\psi:\Gamma\rightarrow \mathbb{R}$ be the normal variation speed at $s=0$ of the family $\Gamma_s \subset \mathbb{S}^n$. Furthermore, we denote with
  $$\pi_{\Gamma_0}:\R^{n+1}\rightarrow\Gamma_0$$ 
 the composition of the closest point projection $\pi_{C_0}$ to $C(\Gamma_0)$ composed with the projection $C(\Gamma_0)\rightarrow \Gamma_0$ on the link. 
 
 \begin{lemma}\label{lem: first jacobi estimate}
 Let $v$ be the Jacobi field induced on $\Sigma_0$ by the above variation. Then we can write
 $$ v = r \cdot \psi \circ \pi_{\Gamma_0} + w $$
 where $r$ is the ambient radius function and $w$ satisfies
 $$ |\nabla ^l w| \leq \frac{c}{r^{1+l}} \qquad \text{for} \ l =0,1\, .$$
  \end{lemma}
 \begin{proof} We choose a family of local parametrisations $(F_s)_{-\varepsilon < s< \varepsilon}$ of the cones $C_s$. We can assume that this family of parametrisation moves in normal direction along the cones, i.e.
 $$\frac{\partial}{\partial s} \bigg|_{s=0} F_s = r\, \tilde{\psi} \, \nu_{C_0} $$
 where $\tilde{\psi}(x) = \psi(r^{-1} x)$ and $\nu_{C_0}$ is a choice of unit normal vector field along $C_0$. Note that a parametrisation of $\Sigma_s$ is then given by
 $$ x \mapsto F_s(x) + u_s(F_s(x)) \nu_{C_s}(x) $$
 where the graph functions $u_s: C(\Gamma_s) \setminus B_{R_0} \rightarrow \mathbb{R}$ satisfy
 $$ \big|\nabla^l_{\Sigma_s} u_s \big| \leq \frac{c}{r^{1+l}} \qquad \text{for} \ l =0,1\, .$$
 Furthermore, we can assume that the variation $w:= \tfrac{d}{ds}|_{s=0} u$ similarly satisfies
 $$ \big|\nabla^l_{\Sigma_0}  w \big| \leq \frac{c}{r^{1+l}} \qquad \text{for} \ l =0,1\, .$$
 Note that by interpolation inequalities from Appendix \ref{app:interpolate}, we also have for any $\delta >0$ that
 $$ \big|\nabla^l_{\Sigma_s} u_s \big| +  \big|\nabla^l_{\Sigma_s} w \big| \leq \frac{c(l,\delta) }{r^{1+l -\delta}}   \qquad {\text{for} \ 2 \leq l \leq 4} \, .$$
Denoting $u:= u_0$, $C:=C_0$, the above implies that we can compute the variation vector field $X$ along $\Sigma_0$ as
$$ X = (r\, \tilde{\psi} + w)\cdot  \nu_{C} - u\cdot  \nabla_{C} (r\, \tilde{\psi}) \ .$$
Working now at a point $x_0 \in C$ and $y_0 = x_0 + u(x_0) \nu_{C}$, we have again, see \eqref{eq:ng.1}, that
 $$ \nu_{\Sigma_0}(y_0) = - \nabla_{\Sigma_0}u (x_0)+ \nu_{C}(x_0) + Q_0(x_0,u,\nabla u) $$
where $|Q_0(p,u,Du)| \leq C(r^{-1}|u| + |\nabla u|^2)$. This implies that
$$v = \big\langle X, \nu_{\Sigma_0}\big\rangle = r_{C} \tilde{\psi} + w + \big\langle  \nabla_{C} (r\, \tilde{\psi}),  \nabla_{C} u \big\rangle\, u+ Q_1$$
where $Q_1 = \textit{O} \left(( r^{-1} |u| + |\nabla u|^2) (|r \tilde{\psi}| + |w| + |u|)\right)$ and $r_C = r \circ \pi_C$. Note further that
$$ r_C = r + O(r^{-1} u^2)\ ,$$    
along $\Sigma_0$. Together with the previous estimate, and replacing $w$ by $w + \big\langle  \nabla_{C} (r\, \tilde{\psi}),  \nabla_{C} u \big\rangle\, u+ Q_1$ implies the statement.
\end{proof}
The next lemma estimates more carefully the hessian of the first approximation of a Jacobi field as in Lemma \ref{lem: first jacobi estimate}. We denote with $U_{\delta}(C)$ the $\delta$-neighborhood of $C$ in $\R^{n+1}$.

\begin{lemma}\label{lem: hessian estimate}
 Let $ \tilde{\psi}:= \psi \circ \pi_{\Gamma_0}$ as above. Then for $l=0,1,2$, one has:
 $$ |D ^l \tilde{\psi}| \leq \frac{c}{r^{l}}\ .$$
 Let $x \in U_{\delta}(C)$ and $x_0 = \pi_C(x)$ denote with $e_0= \nu_c(x_0)$ and $e_1=\frac{x_0}{|x_0|}$. Then
 $$ D_i\tilde{\psi} =0 $$
 for $i=0,1$ as well as
 $$ D^2  \tilde{\psi}(x)(u,v) = 0$$
 for $u,v \in \text{span}\{e_0,e_1\}$ and
 $$ \big|D^2  \tilde{\psi}(x)(u,v)\big| \leq \frac{c}{r^2} |u| |v|$$
 for $v \in  (\text{span}\{e_0,e_1\})^\perp$.
  \end{lemma}
  
\begin{proof}
We choose normal coordinates $F_\Gamma:\mathcal{U}_0\subset\R^{n-1}\rightarrow \Gamma$ of $\Gamma$ around $x_0':=x_0/|x_0|$ and locally parametrise $C(\Gamma)$ via the map $F_C:  (0,\infty) \times \mathcal{U}_0\mapsto C(\Gamma), (r,x') \mapsto r\cdot F_\Gamma(x')$. The induced metric has then the standard form
$$ g_C = dr^2 + r^2 g_\Gamma\, ,$$
where $g_\Gamma$ is the metric in normal coordinates on $\Gamma$. We now parametrise a neighborhood of $C(\Gamma)$ given by $F:(-\delta,\delta)\times \mathcal{U}_0\times (0,\infty) \rightarrow U_\delta(C)$ of $\Sigma_1$ via $(s,r,x') \mapsto F_C(r,x') + s \nu_C\big(F_C(r,x')\big)$.
 Therefore, by taking derivatives in the previous identity we obtain
 \begin{eqnarray*}
\partial_iF=\partial_iF_C-s h_i^j\partial_jF_C\, ,\quad i=1,\ldots,n\, .
\end{eqnarray*}
Thus we obtain for the induced metric $g:=F^*\delta_{\R^{n+1}}$:
\begin{eqnarray*}
g_{ij}=(g_C)_{ij}-2sh_{ij}+s^2h_i^kh_{kj}\quad i,j=1,\ldots,n \, ,
\end{eqnarray*}
and where we omit the $0$-direction since it is orthogonal. Note that $h_{1j}=h_{j1} \equiv 0 $ for $j =1,...,n$ and thus $g_{ij} = g_{ji} = \delta_{ij}$ for $i=0,1;\, j= 0,\ldots, n$. We recall that the Christoffel symbols are given by 
\begin{equation*}
\Gamma_{ij}^k=\frac{1}{2}g^{kl}\left(\partial_ig_{jl}+\partial_jg_{il}-\partial_lg_{ij}\right)\, .
\end{equation*}
We thus obtain 
\begin{equation}\label{eq:christoffel.1}
 \Gamma_{ij}^k \equiv 0\,, \quad i,j \in\{0,1\},\ k\geq 2\ .
\end{equation}
Furthermore, we have for $j,k\geq 2$ the estimates
\begin{equation}\label{eq:christoffel.2}
 |\Gamma_{0j}^k| = |\Gamma_{j0}^k| = \frac{1}{2}\Big| g^{kl}\partial_0 g_{jl}\Big| = \frac{1}{2}\Big|  -2 h_{j}^{\ k} + 2s h_j^{\ r} h_{r}^{\ k}\Big| \leq  \frac{c}{r}
\end{equation}
and
\begin{equation}\label{eq:christoffel.3}
\begin{split}
 \Gamma_{1j}^k  =  \Gamma_{j1}^k &= \frac{1}{2} g^{kl}\partial_1 g_{jl} \\
 &= \frac{1}{2} g^{kl}\bigg( \frac{2}{r} g_{jl} - 2s \partial_0 h_{jl} + s^2 \partial_0 h_i^{\ r}h_{rl} + s^2 h_{i}^{\ r}\partial_0 h_{rl}\bigg)\\
 &= \frac{1}{2} \delta_j^{\ k} - s \partial_0 h_{j}^{\ k} + \frac{2s}{r} h_j^{\ k} + \frac{1}{2} s^2 \partial_0 h_i^{\ r}h_{rl} + \frac{s^2}{2} h_{i}^{\ r}\partial_0 h_{r}^{\ k}\\
 &\ \ \  - \frac{s^2}{r} h_{i}^{\ r}h_{r}^{\ k}\\
 & = \frac{1}{r} \delta_j^{\ k} + O\Big(\frac{s}{r^2}\Big)\ .
 \end{split}
 \end{equation}
A similar argument gives for $i,j,k \geq 2$ that at $(s,r,x_0')$
$$ |\Gamma_{ij}^k | \leq c \frac{s}{r^2}\ . $$
Recall the formula 
$$ D^2 \tilde{\psi} (x) (\partial_i F, \partial_j F) = \partial^2_{ij} (\tilde{\psi}\circ F) - \Gamma_{ij}^k\partial_{k} (\tilde{\psi}\circ F) $$
and note that $\tilde{\psi}\circ F = \psi \circ F_\Gamma$, i.e. it does not depend on the first two coordinates. The above estimates on the Christoffel symbols then imply the stated estimates on $D^2\tilde{\psi}$.
\end{proof}
The following lemma estimates the convergence rate at infinity of the difference of two Jacobi functions associated to two asymptotically conical expanders coming out of the same cone and which comes from the same variation of the cone.

 \begin{lemma}\label{lem: second jacobi estimate}
 Let $(\Sigma_i)_{i=0,1}$ be two expanders asymptotic to the same cone $C(\Gamma)$. Let $(v_i)_{i=0,1}$ be two Jacobi functions associated to $(\Sigma_i)_{i=0,1}$ induced by the same variation of the cone as in Lemma \ref{lem: first jacobi estimate}, i.e. 
 \begin{eqnarray}
\Delta_{\Sigma_i}v_i+\frac{1}{2}\langle x_i,\nabla^{\Sigma_i}v_i\rangle+ |A_{\Sigma_i}|^2v_i-\frac{1}{2}v_i=0,\quad i=0,1.\label{Jac-equ-fct}
\end{eqnarray}
Then, writing $\Sigma_1$ as a normal graph over $\Sigma_0$ via $F(p):= p + u(p) \nu_{\Sigma_0}(p)$, and setting $\tilde{v}_1 = v_1\circ F$,  the difference $\tilde{v}_1 -v_0$ satisfies the following decay: 
\begin{eqnarray}
\tilde{v}_1 -v_0=\textit{O}\left(r^{-n-1+\varepsilon}e^{-\frac{r^2}{4}}\right),\label{diff-jac-fiel-est}
\end{eqnarray}
for every positive $\varepsilon$.
 \end{lemma}
 
 \begin{remark}
 The estimate (\ref{diff-jac-fiel-est}) of Lemma \ref{lem: second jacobi estimate} can be made more precise. Indeed, the proof of Lemma \ref{lem: second jacobi estimate} shows that the difference $\tilde{v}_1 -v_0$ is $\textit{O}\Big(r^{-n-1}e^{-\frac{r^2}{4}}\Big)$ if and only if the mean curvature $H_{\Sigma_0}$ is radially independent at infinity, i.e.~if the mean curvature of the asymptotic cone $C_0$ is radially independent, which forces the asymptotic cone $C_0$ to be minimal. However, the estimate (\ref{diff-jac-fiel-est}) is sufficient for our purpose.
 \end{remark}
 \begin{proof}
 The main difficulty in estimating the difference $v_1-v_0$ of these two Jacobi functions lies in the fact that they are asymptotically $1$-homogeneous in a sense we recall now. By Lemma \ref{lem: first jacobi estimate} we have
  \begin{eqnarray*}
v_i=r \cdot \tilde{\psi} +w_i=:\alpha+w_i,\quad \text{$i=0,1$ on $\Sigma_i$,}
\end{eqnarray*}
where the functions $w_i$ are asymptotically $(-1)$-homogeneous, i.e. $\nabla^l_{\Sigma_i}w_i=\textit{O}(r_i^{-1-l})$ for $l=0,1$. By (\ref{Jac-equ-fct}),
\begin{eqnarray}
&&\Delta_{\Sigma_i}w_i+\frac{1}{2}\langle x_i,\nabla^{\Sigma_i}w_i\rangle+\left(|A_{\Sigma_i}|^2-\frac{1}{2}\right)w_i=\\
&&-\Delta_{\Sigma_i}\alpha-\frac{1}{2}\langle x_i,\nabla^{\Sigma_i}\alpha\rangle-\left(|A_{\Sigma_i}|^2-\frac{1}{2}\right)\alpha.\label{est-wi-fcts}
\end{eqnarray}
Now, recall that for an ambient function $g:\R^{n+1}\rightarrow\R$, the intrinsic Laplacian with respect to the expander $\Sigma_i$, $i=0,1$ and the extrinsic Laplacian are related by:
\begin{eqnarray*}
\Delta_{\Sigma_i}g&=&\Delta_{\R^{n+1}}g+\langle \bfH_i,Dg\rangle-D^2g(\nu_{i},\nu_{i})\\
&=&-\frac{1}{2}\langle x_i,\nabla^{\Sigma_i}g\rangle+\Delta_{\R^{n+1}}g-D^2g(\nu_{i},\nu_{i})+\frac{1}{2}\langle x_i,Dg\rangle.
\end{eqnarray*}
By this observation together with (\ref{est-wi-fcts}), one gets:
\begin{equation*}\begin{split}
\Delta_{\Sigma_i}w_i+&\frac{1}{2}\langle x_i,\nabla^{\Sigma_i}w_i\rangle+\left(|h_{\Sigma_i}|^2-\frac{1}{2}\right)w_i\\
&= -\Delta_{\R^{n+1}}\alpha+D^2\alpha(\nu_i,\nu_i)-\frac{1}{2}\langle x_i,D\alpha \rangle-\left(|A_{\Sigma_i}|^2-\frac{1}{2}\right)\alpha\\
&= -\Delta_{\R^{n+1}}\alpha+D^2\alpha(\nu_i,\nu_i)-|A_{\Sigma_i}|^2\alpha.
\end{split}
\end{equation*}
We are now in a good position to linearize the Jacobi operator $$\Delta_{\Sigma_1} +\frac{1}{2}\langle x_1,\nabla^{\Sigma_1}\cdot\rangle+\left(|A_{\Sigma_1}|^2-\frac{1}{2}\right),$$ with respect to the (background) expander $\Sigma_0$.

To do so, we employ the computations in Appendix \ref{app:NG}. Using the notation there, let $x_0 \in \Sigma_0\setminus B_{R_0}$ where $R_0$ is a sufficiently large radius and choose a local parametrization of $\Sigma_0$ in a neighborhood of $x_0$, $F:\mathcal{U}_0\subset\R^n\rightarrow \Sigma_0$ such that 
\begin{eqnarray*}
F(0)=x_0,\quad \langle\partial_iF(0),\partial_jF(0)\rangle=\delta_{ij},\quad\partial_{ij}^2F(0)=h(\partial_iF(0),\partial_jF(0))\nu(x_0),
\end{eqnarray*}
such that the second fundamental form is diagonal. Note that this implies that the Christoffel symbols if the induced metric vanish at zero.

 In particular, this implies there exists a local parametrization $F_1:\mathcal{U}_0\rightarrow \Sigma_1$ of $\Sigma_1$ in a neighborhood of $y_0=x_0+u(x_0)\nu(x_0)$ such that $$F_1=F_0+u\nu.$$
 Therefore, by taking derivatives in the previous identity:
 \begin{eqnarray*}
\partial_iF_1=\partial_iF_0+\partial_iu\nu-uh_i^j\partial_jF_0,\quad i=1,\ldots,n.
\end{eqnarray*}

Recall that the induced metric on $\Sigma_1$ is given by $F_1^*g_{\R^{n+1}}$:
\begin{eqnarray*}
(g_1)_{ij}=(g_0)_{ij}-2uh_{ij}+\partial_iu\partial_ju+u^2h_i^kh_{kj},
\end{eqnarray*}
since $\langle\nu,\partial_iF_0\rangle=0$.
 Consequently, one has the following schematic estimate for the difference of the two metrics $g_1-g_0$:
 \begin{eqnarray}
(g_1)_{ij}-(g_0)_{ij}=-2uh_{ij}+\textit{O}\left(|\nabla^{\Sigma_0} u|^2+|u|^2r^{-2}\right).\label{est-diff-metrics}
\end{eqnarray}
It remains to estimate the difference of the Christoffel symbols $(\Gamma_1)_{ij}^k-(\Gamma_0)_{ij}^k$, with obvious notations. By definition:
\begin{eqnarray*}
(\Gamma_1)_{ij}^k&=&\frac{1}{2}g_1^{kl}\left(\partial_i(g_1)_{jl}+\partial_j(g_1)_{il}-\partial_l(g_1)_{ij}\right),
\end{eqnarray*}
and:
\begin{eqnarray}
\partial_k(g_1)_{ij}&=&\partial_k(g_0)_{ij}-2\partial_kuh_{ij}-2u\partial_kh_{ij}\label{est-1-partial-der}\\
&&+\textit{O}\left(|u|^2r^{-3}+|u||\nabla^{\Sigma_0}u|r^{-2}+|\nabla^{\Sigma_0,2}u||\nabla^{\Sigma_0}u|\right).\label{est-2-partial-der}
\end{eqnarray}
Therefore, after some tedious computations, one arrives at:
\begin{eqnarray*}
(\Gamma_1)_{ij}^k&=&-u\partial_ih^k_j-u\partial_jh^k_i+u\partial_kh_{ij}-\partial_iuh^k_j-\partial_juh^k_i+\partial_kuh_{ij}\\
&&+\textit{O}\left(|u|^2r^{-3}+|u||\nabla^{\Sigma_0}u|r^{-2}+|\nabla^{\Sigma_0,2}u||\nabla^{\Sigma_0}u|\right).\label{est-chris-symb}
\end{eqnarray*}
By using the definition of the Laplacian acting on functions:
\begin{eqnarray*}
\Delta_{\Sigma_1}w_1&=&g_1^{ij}\left(\partial^2_{ij}w_1-(\Gamma_1)_{ij}^k\partial_kw_1\right)\\
&=&\Delta_{\Sigma_0}w_1+2A_{\Sigma_0}(\nabla^{\Sigma_0}u,\nabla^{\Sigma_0}w_1)+u\langle\nabla^{\Sigma_0}H_{\Sigma_0},\nabla^{\Sigma_0}w_1\rangle\\
&&-H_{\Sigma_0}\langle\nabla^{\Sigma_0}u,\nabla^{\Sigma_0}w_1\rangle+2u\langle A_{\Sigma_0},\partial^2w_1\rangle\\
&&+\textit{O}\left(|u|^2r^{-3}+|u||\nabla^{\Sigma_0}u|r^{-2}+|\nabla^{\Sigma_0,2}u||\nabla^{\Sigma_0}u|\right)\\
&=&\Delta_{\Sigma_0}w_1+2A_{\Sigma_0}(\nabla^{\Sigma_0}u,\nabla^{\Sigma_0}w_1)+u\langle\nabla^{\Sigma_0}H_{\Sigma_0},\nabla^{\Sigma_0}w_1\rangle\\
&&-H_{\Sigma_0}\langle\nabla^{\Sigma_0}u,\nabla^{\Sigma_0}w_1\rangle+\textit{O}(|u|r^{-4})\\
&&+\textit{O}\left(|u|^2r^{-3}+|u||\nabla^{\Sigma_0}u|r^{-2}+|\nabla^{\Sigma_0,2}u||\nabla^{\Sigma_0}u|\right),
\end{eqnarray*}
where in the penultimate line, we used $\partial^2w_1=\textit{O}(r^{-3})$. Now, let us estimate the linear terms in $u$ more carefully by considering $\hat{u}$ instead of $u$:
\begin{eqnarray*}
A_{\Sigma_0}(\nabla^{\Sigma_0}u,\nabla^{\Sigma_0}w_1)&=&A_{\Sigma_0}(\nabla^{\Sigma_0}\hat{u},\nabla^{\Sigma_0}w_1)r^{-n-1}e^{-\frac{r^2}{4}}\\
&&-\left(\frac{r}{2}+\frac{n+1}{r}\right)h(\textbf{n},\nabla^{\Sigma_0}w_1)u,
\end{eqnarray*}
where $$\textbf{n}:=\frac{x^{\top}}{|x^{\top}|}.$$
Using \eqref{eq.0.1} we obtain
\begin{eqnarray*}
|A_{\Sigma_0}(\nabla^{\Sigma_0}u,\nabla^{\Sigma_0}w_1)|&\leq&\frac{c}{r^3}|\nabla^{\Sigma_0}\hat{u}|r^{-n-1}e^{-\frac{r^2}{4}}+c\frac{|u|}{r^3},
\end{eqnarray*}
for some positive constant $c$ uniform in $r\geq R_0$. One gets a similar estimate on the term $u\langle\nabla^{\Sigma_0}H_{\Sigma_0},\nabla^{\Sigma_0}w_1\rangle$. All in all, one arrives at:
\begin{eqnarray}
&&|\Delta_{\Sigma_1}w_1-\Delta_{\Sigma_0}w_1|\leq\label{lin-lap-1}
\\
&& \textit{O}\left(|u|r^{-3}+|\nabla^{\Sigma_0}\hat{u}|r^{-n-4}e^{-\frac{r^2}{4}}+|\nabla^{\Sigma_0,2}u||\nabla^{\Sigma_0}u|\right). \label{lin-lap-2}
\end{eqnarray}
Let us now linearize the drift term as follows:
\begin{eqnarray*}
\frac{1}{2}\langle x_1,\nabla^{\Sigma_1}w_1\rangle&=&\frac{1}{2}\langle x_0+u\nu,(g_1)^{ij}\partial_iw_1\partial_jF_1\rangle.
\end{eqnarray*}
Now,
\begin{eqnarray*}
(g_1)^{ij}\partial_iw_1\partial_jF_1&=&\left((g_0)^{ij}+2uh_{ij}\right)\partial_iw_1\left(\partial_jF_0+\partial_ju\nu-uh_j^k\partial_kF_0\right)\\
&&+\textit{O}\left(|u|^2r^{-4}+|\nabla^{\Sigma_0}u|^2r^{-2}\right),
\end{eqnarray*}
where we used the quadratic decay of $\partial_iw_1$ in the last line.
Consequently,
\begin{eqnarray*}
\frac{1}{2}\langle x_1,\nabla^{\Sigma_1}w_1\rangle&=&\frac{1}{2}\langle x_0,\nabla^{\Sigma_0}w_1\rangle+\frac{1}{2}\langle\nabla^{\Sigma_0}u,\nabla^{\Sigma_0}w_1\rangle\langle x_0,\nu\rangle\\
&&+\frac{1}{2}uA_{\Sigma_0}(\nabla^{\Sigma_0}w_1,x_0^{\top})\\
&&+\textit{O}(|u||\nabla^{\Sigma_0}u|r^{-4}+|u|^2r^{-3}+|\nabla^{\Sigma_0}u|^2r^{-1}),
\end{eqnarray*}
which implies by using (\ref{eq.0.1}):
\begin{eqnarray}
\left|\frac{1}{2}\langle x_1,\nabla^{\Sigma_1}w_1\rangle-\frac{1}{2}\langle x_0,\nabla^{\Sigma_0}w_1\rangle\right|\leq \textit{O}\left(|u|r^{-4}+ |\nabla^{\Sigma_0}u|r^{-3}\right).\label{lin-lap-3}
\end{eqnarray}
We summarize this discussion by adding (\ref{lin-lap-1}), (\ref{lin-lap-2}) and (\ref{lin-lap-3}) to get:
\begin{equation*}
\begin{split}
\Delta_{\Sigma_1}w_1+\frac{1}{2}\langle x_1,\nabla^{\Sigma_1}w_1\rangle=&\ \Delta_{\Sigma_0}w_1+\frac{1}{2}\langle x_0,\nabla^{\Sigma_0}w_1\rangle\\
&+\textit{O}\left(|u|r^{-3}+|\nabla^{\Sigma_0}u|r^{-3}+|\nabla^{\Sigma_0}\hat{u}|r^{-n-4}e^{-\frac{r^2}{4}}\right)\\
=&\ \Delta_{\Sigma_0}w_1+\frac{1}{2}\langle x_0,\nabla^{\Sigma_0}w_1\rangle\\
&+\textit{O}\left(|u|r^{-2}+|\nabla^{\Sigma_0}\hat{u}|r^{-n-4}e^{-\frac{r^2}{4}}\right)\\
\end{split}
\end{equation*}
where we used the fact that $\nabla^{\Sigma_0,2}u=\textit{O}(r^{-3})$.

Going back to (\ref{est-wi-fcts}), the difference $v:=v_1-v_0$ satisfies:
\begin{equation}\label{eq: first diff}\begin{split}
\Delta_{\Sigma_0}v+\frac{1}{2}\langle x_0,\nabla^{\Sigma_0}v\rangle+&\left(|A_{\Sigma_0}|^2-\frac{1}{2}\right)v\bigg|_{x=x_0}\\
&=\Delta_{\R^{n+1}}  \alpha (x_0) - \Delta_{\R^{n+1}}  \alpha (y_0)\\ &\ \  + D^2\alpha(\nu_1,\nu_1)(y_0) -D^2\alpha(\nu_0,\nu_0)(x_0))\\
&\ \ -\left(|A_{\Sigma_1}|^2-|A_{\Sigma_0}|^2\right)\alpha\\
&\ \ +\textit{O}\left(|u|r^{-2}+|\nabla^{\Sigma_0}\hat{u}|r^{-n-4}e^{-\frac{r^2}{4}}\right).
\end{split}
\end{equation}
To estimate the terms on the right hand side, it is easy to check that
$$ |D^l\alpha| \leq \frac{c}{r^{l-1}}\ .$$
This implies that
$$  \Delta_{\R^{n+1}}  \alpha (x_0) - \Delta_{\R^{n+1}}  \alpha (y_0) = O(r^{-2} |u|)\, .$$
We now again use that
$$ \nu_1(y_0) = \nu_0(x_0)  - \nabla_{\Sigma_0}u (x_0) + Q_0(x_0,u,\nabla u) $$
where $|Q_0(p,u,Du)| \leq C(r^{-1}|u| + |\nabla u|^2)$. We first note that
$$ D^2 \alpha = \tilde{\psi} D^2 r + Dr \otimes D\tilde{\psi} + D\tilde{\psi}\otimes Dr + r D^2 \tilde{\psi}\ .$$
Combining this with the estimates in Lemma \ref{lem: hessian estimate} we can estimate after a longer calculation that
\begin{equation*}
\begin{split}
D^2\alpha(\nu_1,\nu_1)(y_0) -D^2\alpha(\nu_0,\nu_0)(x_0) &=  O\big(r^{-2} |u| + r^{-1} \big|(\nabla u)^{\perp_r}\big| + r^{-3} |\nabla u|\big)\ ,
\end{split}
\end{equation*}
where $\perp_r$ is the orthogonal projection on the non-radial directions. Combining this with \eqref{eq: first diff} we see that
\begin{equation}
\begin{split}
\Delta_{\Sigma_0}v+&\frac{1}{2}\langle x,\nabla^{\Sigma_0}v\rangle +\left(|A_{\Sigma_0}|^2-\frac{1}{2}\right)v \\
&= \textit{O}\bigg(r^{-2}|u| +  r^{-1} \big|(\nabla u)^{\perp_r}\big| + r^{-3} |\nabla u|+ |\nabla^{\Sigma_0}\hat{u}|r^{-n-4}e^{-\frac{r^2}{4}}\bigg)\\
&\ \ \ -\left(|A_{\Sigma_1}|^2-|A_{\Sigma_0}|^2\right)\alpha\\
&=\textit{O}\left(r^{-n-3}e^{-\frac{r^2}{4}}\right)-\left(|A_{\Sigma_1}|^2-|A_{\Sigma_0}|^2\right)\alpha\,,
\end{split}
\end{equation}
where in the last line we used estimates on $u$ and the gradient of $\hat{u}$ from Theorem \ref{theo-ptwise-est}: notice the crucial presence of the orthogonal projection of $\nabla^{\Sigma_0}u$ perpendicular to the radial direction.

It remains to estimate the difference of the squared norms of the respective second fundamental forms $A_{\Sigma_i}$ for $i=0,1$. Again, we use extensively Appendix \ref{app:NG} to estimate this difference: our goal is to show that it decays as fast as $r^{-n-2}e^{-r^2/4}$ in order to apply the second part of Theorem \ref{theo-ptwise-est} with $Q=\textit{O}\Big(r^{-n-1}e^{-\frac{r^2}{4}}\Big)$ since $\alpha$ grows linearly.

 Thanks to (\ref{eq:ng.4}), one can write:
\begin{equation*}
\begin{split}
|A_{\Sigma_1}|^2-|A_{\Sigma_0}|^2=&\ 2\langle A_{\Sigma_0}\otimes A_{\Sigma_0},A_{\Sigma_0}\rangle u+2\langle\nabla^{\Sigma_0,2}u,A_{\Sigma_0}\rangle\\
&+Q(x,u,\nabla^{\Sigma_0} u , \nabla^{\Sigma_0,2} u),
\end{split}
\end{equation*}
where $Q$ is quadratic in $u$ and its derivatives. Now,
\begin{equation*}
\begin{split}
\langle A_{\Sigma_0}\otimes A_{\Sigma_0},A_{\Sigma_0}\rangle u&=\textit{O}\Big(r^{-n-4}e^{-\frac{r^2}{4}}\Big),\\
\langle\nabla^{\Sigma_0,2}u,A_{\Sigma_0}\rangle&=\Big\langle\nabla^{\Sigma_0,2}\Big(r^{-n-1}e^{-\frac{r^2}{4}}\hat{u}\Big),A_{\Sigma_0}\Big\rangle\\
&=r^{-n-1}e^{-\frac{r^2}{4}}\Big\langle\nabla^{\Sigma_0,2}\hat{u},A_{\Sigma_0}\Big\rangle\\
&\ \ \ -\left(r+\textit{O}(r^{-1})\right)r^{-n-1}e^{-\frac{r^2}{4}} A_{\Sigma_0}(\textbf{n},\nabla^{\Sigma_0}\hat{u})\\
&\ \ \ +\left(\frac{r^2}{4}+\textit{O}(1)\right)r^{-n-1}e^{-\frac{r^2}{4}}\cdot\hat{u}\cdot A_{\Sigma_0}(\textbf{n},\textbf{n}).
\end{split}
\end{equation*}
According to Theorem \ref{theo-ptwise-est}, $\nabla^{\Sigma_0,i}\hat{u}=\textit{O}(r^{-i})$, $i=0,1,2$, and together with (\ref{eq.0.1}), one gets:
\begin{equation*}
\begin{split}
&\langle\nabla^{\Sigma_0,2}u,A_{\Sigma_0}\rangle=\textit{O}\Big(r^{-n-4}e^{-\frac{r^2}{4}}\Big)+\textit{O}\Big(r^{-n-1}e^{-\frac{r^2}{4}}\Big)\cdot\hat{u}\cdot \left( r^2A_{\Sigma_0}(\textbf{n},\textbf{n})\right).
\end{split}
\end{equation*}
In general, we know that $r^3A_{\Sigma_0}(\textbf{n},\textbf{n})=\textit{O}(1)$ by (\ref{eq.0.1}). Notice that if $$\lim_{r\rightarrow +\infty}r^3A_{\Sigma_0}(\textbf{n},\textbf{n})=0,$$ then $H_{\Sigma_0}$ is asymptotically radially independent by (\ref{eq.0.1}), i.e. $H_{C_0}$ vanishes identically on the asymptotic cone $C_0$. 

Since $v$ satisfies the assumption \eqref{app-jac-fiel} with $Q=\textit{O}\Big(r^{-n-1}e^{-\frac{r^2}{4}}\Big)$, we know that $v=\textit{O}\Big(r^{-n-1+\varepsilon}e^{-\frac{r^2}{4}}\Big)$ by invoking the second part of Theorem \ref{theo-ptwise-est} for every positive $\varepsilon$.
\end{proof} 

\section{Renormalization of the expander entropy}\label{ren-exp-ent-section}
For $R \geq R_1$  we now define the approximate relative entropy 
 $$\calE_{\Sigma_0,\Sigma_1}(R):= \int_{\Sigma_1 \cap B_R(0)} e^\frac{r^2}{4}\, d\calH^n - \int_{\Sigma_0 \cap B_R(0)} e^\frac{r^2}{4}\, d\calH^n\ .$$
We aim to show that $\calE_{\Sigma_0,\Sigma_1}(R)$ converges as $R \rightarrow \infty$. It turns out that using cut-off functions is more convenient to control the asymptotic behavior of the integrals involved in the definition of this relative entropy, as the following proposition demonstrates.

\begin{proposition}\label{thm:1.3} For $R_0<R_1<R_2 -2$ and $0<\delta<1$, let $\varphi: [0,\infty) \ra [0,1]$  be a smooth cut-off function, compactly supported in $(R_1, R_2)$, such that 
\begin{eqnarray*}
\text{$\varphi \equiv 1$ on $[R_1+\delta, R_2 -\delta]$, $|\varphi'| \leq (1+\delta)\delta^{-1}$ and $|\varphi''| \leq (1+\delta) \delta^{-2}$.}
\end{eqnarray*}
 Then
\begin{equation}
 \left|\int_{\Sigma_1} \varphi(r)\, e^\frac{r^2}{4}\, d\calH^n - \int_{\Sigma_0} \varphi(r)\, e^\frac{r^2}{4}\, d\calH^n\right|  \leq C R_1^{-3}\ ,
\end{equation}
where $C$ is independent of $\delta$.
\end{proposition}
\begin{proof}
 We consider for $s \in [0,1]$ the family
$$ \Sigma'_s = \text{graph}_{\bar{E}_{R_0,0}}(s u)\, .$$
which has the variation vector field
\begin{equation}\label{eq.8}
Y(x,s) = u(\pi(x))\, \nu_{\Sigma_0}(\pi(x)),
\end{equation}
where $\pi$ is the nearest point projection onto $\Sigma_0$. Note that
\begin{equation}\label{eq.8b}
Z:= \frac{\partial}{\partial s} Y \equiv 0\ .
\end{equation}
We compute 
\begin{equation}\label{eq.8c}
\begin{split}
 \frac{d}{ds}\int_{\Sigma'_s}\varphi \, e^\frac{r^2}{4}&\, d\mu = \int_{\Sigma'_s} \left(\langle D\varphi, Y\rangle + \varphi\, \text{div}(Y) + \varphi \frac{\langle x, Y\rangle}{2}\right) e^\frac{r^2}{4}\, d\calH^n\\
&= \int_{\Sigma'_s} \text{div}\Big(\varphi\, e^\frac{r^2}{4}Y\Big) + \left(\langle D\varphi, Y^\perp\rangle + \varphi \frac{\langle x^\perp, Y\rangle}{2}\right)  e^\frac{r^2}{4}\, d\calH^n\\
&= \int_{\Sigma'_s} \varphi \left\langle \frac{x^\perp}{2} -\bfH, Y\right \rangle  e^\frac{r^2}{4}\, d\calH^n + \int_{\Sigma'_s} \langle D\varphi, Y^\perp\rangle e^\frac{r^2}{4}\, d\calH^{n}\, .
 \end{split}
 \end{equation}
 Note that by \eqref{eq.0} and \eqref{eq.8}  we have along $\Sigma_0$ that $ |\langle D\varphi , Y^\perp \rangle | \leq C \delta^{-1} r^{-2} |u| $. 
 
 Denoting $S_{0,\rho} : = \Sigma_0 \cap \mathbb{S}_\rho$, we see with [(\ref{ptwise-est-u}), Theorem \ref{theo-ptwise-est}] for $i=0$ that
\begin{equation*}
\begin{split}
 \left| \int_{\Sigma_0} \langle D\varphi, Y^\perp \rangle\, e^\frac{r^2}{4}\, d\calH^{n} \right| &\leq C \delta^{-1} \int_{R_1}^{R_1+\delta} \rho^{-3-n} \int_{S_{0,\rho}} |\hat{u}| \, d\calH^{n-1}\, d\rho \\
 &\ \ + C \delta^{-1} \int_{R_2-\delta}^{R_2} \rho^{-3-n} \int_{S_{0,\rho}} |\hat{u}| \, d\calH^{n-1}\, d\rho\\
 &\leq C R_1^{-4} \sup_{\rho \geq R_1}\left(\rho^{1-n} \int_{S_{0,\rho}} |\hat{u}|^2\, d\calH^{n-1}\right)^\frac{1}{2}\\
 & \leq C R_1^{-4}\sup_{\rho \geq R_1}|\hat{u}|\leq CR_1^{-4}\, ,
 \end{split}
 \end{equation*}
where we used the decay estimate on $u$ in the penultimate inequality and where $C$ is a positive constant independent of $\delta>0$.  Thus we obtain
\begin{equation}\label{eq.9}
  \frac{d}{ds}\bigg|_{s=0}\int_{\bar{E}_{s,R,\rho}}e^\frac{r^2}{4}\, d\calH^n =  O(R^{-4})\ .
  \end{equation}
 To compute the second derivative we use equation (9.4) in \cite{Simon} together with \eqref{eq.8b} to get
 \begin{equation*}
 \begin{split}
 \frac{d^2}{ds^2}\int_{\Sigma'_s}\varphi\, e^\frac{r^2}{4}\, d\calH^n =& \int_{\Sigma'_s} D^2\varphi(Y,Y)\, e^\frac{r^2}{4}\, d\calH^n\\
 &+ \int_{\Sigma'_s} \bigg(\langle D\varphi, Y\rangle \left( \text{div}(Y) + \frac{\langle x, Y\rangle}{2} \right)\bigg)e^\frac{r^2}{4}\, d\calH^n \\
 & + \int_{\Sigma'_s} \varphi \bigg(\left(\text{div}(Y) + \frac{\langle x, Y\rangle}{2}\right)^2 + \frac{|Y|^2}{2} \\
 &+ \sum_{i=1}^n |(D_{\tau_i}Y)^\perp|^2 - \sum_{i,j=1}^n \langle \tau_i , D_{\tau_j}Y\rangle \langle \tau_j , D_{\tau_i}Y\rangle\bigg) e^\frac{r^2}{4}\, d\calH^n \\
 &=: E+F+G\, ,
 \end{split}
 \end{equation*}
 where $(\tau_i)_{i=1}^n$ is a orthonormal frame along $\Sigma'_s$. Note that by \eqref{eq.8} and since $|D\pi|$ is uniformly bounded, we have
\begin{align*}
|DY|(x) &\leq C \big(|\nabla_{\Sigma_0} u|(\pi(x)) + r^{-1} |u|(\pi(x))\big).
\end{align*}
 We can thus estimate that for $s \in [0,1]$, using [(\ref{ptwise-est-u}), Theorem \ref{theo-ptwise-est}] for $i=0,1$,
 \begin{equation*}
 \begin{split}
  |G| & \leq C\int_{\Sigma'_0}\varphi \left(u^2+|\nabla u|^2\right)  e^\frac{r^2}{4}\, d\calH^n \\
  &\leq R_1^{-n} e^{-\frac{R_1^2}{4}} \int_{\bar{E}_{0,R_1}}\left( \hat{u}^2+ r^2 |\nabla \hat{u}|^2\right)  r^{-n-1}\, d\calH^n\\
  & \leq C R_1^{-n} e^{-\frac{R_1^2}{4}}\, .
  \end{split}
 \end{equation*}
 Similarly as before, we have for the second term, using \eqref{eq.0}
 $$\left|\langle D\varphi, Y\rangle \left( \text{div}(Y) + \frac{\langle x, Y\rangle}{2} \right)\right| \leq C |\varphi'| |u|(|\nabla_{\Sigma_0}u|+ r |u|) \leq C|\varphi'|(r^{-2}|u|+ r |u|^2)$$
 and thus,  denoting $S_{s,\rho} : = \mathbb{S}_\rho\cap \Sigma'_s$ and using that $\hat{u}$ is uniformly bounded, we can estimate
  \begin{equation*}
 \begin{split}
 |F| &\leq C \delta^{-1}\bigg( \int_{R_1}^{R_1+\delta} \rho^{-3-n} \calH^{n-1}(S_{s,\rho}) + \rho^{-1-2n}e^{-\frac{\rho^2}{4}}  \calH^{n-1}(S_{s,\rho})   \, d\rho\bigg)
   \\
 &\ \ + C \delta^{-1} 
 \bigg(\int_{R_2-\delta}^{R_2} \rho^{-3-n}\calH^{n-1}(S_{s,\rho}) + \rho^{-1-2n}e^{-\frac{\rho^2}{4}}  \calH^{n-1}(S_{s,\rho})   \, d\rho\bigg)\\
 &\leq C R_1^{-4} + C R_1^{-2-n}e^{-\frac{R_1^2}{4}} \leq C R_1^{-4}\, .
 \end{split}
 \end{equation*}
 For the first term we note that
 $$D^2\varphi(Y,Y) = \varphi''  \langle Dr,Y\rangle^2 + \varphi' D^2r(Y,Y)$$
 and we can estimate as just before
  \begin{equation*}
 \begin{split}
 \int_{\Sigma'_s}\left|\varphi' D^2r(Y,Y)\right| e^\frac{r^2}{4}\, d\calH^n &\leq C \delta^{-1} \int_{R_1}^{R_1+\delta} \rho^{-3-2n}e^{-\frac{\rho^2}{4}}  \calH^{n-1}(S_{s,\rho})  \, d\rho
   \\
 &\ \ + C \delta^{-1} \int_{R_2-\delta}^{R_2} \rho^{-3-2n}e^{-\frac{\rho^2}{4}}  \calH^{n-1}(S_{s,\rho})\, d\rho\\
 &\leq CR_1^{-4}\, .
\end{split}
 \end{equation*}
 For the other term we note that we can integrate by parts as follows
  \begin{equation*}
 \begin{split}
  \int_{\Sigma'_s} \varphi''  \langle Dr,Y\rangle^2 e^\frac{r^2}{4}\, d\calH^n & =  \int_{R_1}^{R_2} \varphi''(\rho) 
  \int_{S_{s,\rho}} \frac{ \langle Dr,Y\rangle^2}{|\nabla_{\Sigma_s} r|}  e^\frac{r^2}{4}\, d\calH^{n-1}\, d\rho\\
  & = - \int_{R_1}^{R_2} \varphi'(\rho) 
\frac{\partial}{\partial \rho}\left(  \int_{S_{s,\rho}} \frac{ \langle Dr,Y\rangle^2}{|\nabla_{\Sigma_s} r|}  e^\frac{r^2}{4}\, d\calH^{n-1}\right) d\rho \, .
  \end{split}
 \end{equation*}
 Note that for every fixed $s$ the hypersurfaces $S'_{s,\rho} = \rho^{-1} \cdot S_{s,\rho} \subset \mathbb{S}^n$ converge in the $C^4$ topology to the link $\Sigma$, and this convergence is uniform in $s$. We write
 \begin{equation*}
 I(s,\rho):=  \int_{S_{s,\rho}} \frac{ \langle Dr,Y\rangle^2}{|\nabla_{\Sigma_s} r|}  e^\frac{r^2}{4}\, d\calH^{n-1} =  \rho^{n-1} \int_{S'_{s,\rho}} \frac{ \langle Dr(\rho \theta),Y(\rho \theta)\rangle^2}{|\nabla_{\Sigma_s} r(\rho \theta) |}  e^\frac{\rho^2}{4}\, d\calH^{n-1}(\theta)\, .
 \end{equation*}
Thus we can estimate
 \begin{equation*}
 \begin{split}
 \left| \frac{\partial}{\partial \rho} I \right| &\leq C \int_{S_{0,\rho}} \left(\rho |u|^2 + \rho^{-1} |u| \right)  e^\frac{\rho^2}{4}\, d\calH^{n-1} \\
 & \leq C \rho^{-n-2} e^{-\frac{\rho^2}{4}} \int_{S_{0,\rho}} |\hat{u}|^2\, d\calH^{n-1} + C \rho^{-3} \left(\rho^{1-n} \int_{S_{0,\rho}} |\hat{u}|^2\, d\calH^{n-1}\right)^\frac{1}{2}\\
 &\leq C \rho^{-3}\, .
   \end{split}
 \end{equation*}
 As before, this implies
 $$ |E| \leq C R_1^{-3} \, .$$
Combining these estimates we see that 
 \begin{equation*}
 \begin{split}
 \int_{\Sigma_1} \varphi\, e^\frac{r^2}{4}\, d\calH^n - \int_{\Sigma_0} \varphi\, e^\frac{r^2}{4}\, d\calH^n = O(R_1^{-3})\ .
\end{split}
\end{equation*}
\end{proof}

\begin{corollary}\label{thm:1.4} Let $\Sigma_0, \Sigma_1$ be two expanders asymptotic to the cone $C(\Gamma)$. Then the relative entropy 
 $$ \calE_{\Sigma_0,\Sigma_1}:= \lim_{R\ra \infty} \calE_{\Sigma_0,\Sigma_1}(R)$$
 is well defined.
\end{corollary}
 We turn to the differentiability of the relative entropy that is now well-defined by Corollary \ref{thm:1.4}:
 
\begin{theorem}\label{thm:1.5} 
Let  $(\Gamma_s)_{-\eps<s<\eps}$ be a continuously differentiable  family of $C^5$ hypersurfaces of $\mathbb{S}^n$. Assume that $(\Sigma_{0,s}, \Sigma_{1,s})_{-\eps<s<\eps}$ is a continuously differentiable family of expanders such that both $\Sigma_{0,s}, \Sigma_{1,s}$ are asymptotic to $C_s:=C(\Gamma_s)$. We assume further, that the normal parts of the corresponding variation vectorfields $Y^\perp_{0,s}, Y^\perp_{1,s}$ are  asymptotic to the variation vectorfield
 $$ Z(s) = r \psi(\theta,s) \nu_s,$$
 where $\nu_s$ is a choice of continuous normal vectorfield of $C_s$ and $\theta \in \Gamma_s $. Then
 $\calE_{\Sigma_{0,s},\Sigma_{1,s}}$ is differentiable in $s$ with
 \begin{equation}\label{eq.10}
 \frac{d}{ds}\bigg|_{s=0} \calE_{\Sigma_{0,s},\Sigma_{1,s}} = -\frac{1}{2} \int_\Gamma \psi \tr_{\infty}^0(\hat{u})\, d\calH^{n-1}\, ,
 \end{equation}
where $\tr_{\infty}^0(\hat{u})$ is the trace of $\hat{u}$ at infinity, as defined in Corollary \ref{coro-def-trace-inf}.
 
 \end{theorem}
 
\begin{proof}
In order to lighten the notation, we denote the trace of $\hat{u}$ at infinity by $\hat{a}:=\tr_{\infty}^0(\hat{u})$.

 For $R > R_0+2$ and $0<\delta<1$ let $\varphi: \R \ra [0,1]$  be a smooth cut-off function, compactly supported in $(-\infty, 0]$, such that $\varphi \equiv 1$ on $(-\infty, -\delta]$, $|\varphi'| \leq (1+\delta)\delta^{-1}$ and $|\varphi''| \leq (1+\delta) \delta^{-2}$. We let $\varphi_R:= \varphi(r-R)$ and define
 $$\calE_{\Sigma_{0,s},\Sigma_{1,s}, \varphi_R}:= \int_{\Sigma_{1,s}} \varphi_R\, e^\frac{r^2}{4}\, d\calH^n - \int_{\Sigma_{0,s}} \varphi_R\, e^\frac{r^2}{4}\, d\calH^n\, .$$
 
 Computing as in \eqref{eq.8c} we see that
\begin{equation}\label{eq:10b}
 \begin{split}
\frac{d}{ds}\calE_{\Sigma_{0,s},\Sigma_{1,s}, \varphi_R} =  \int_{\Sigma_{1,s}} \langle D\varphi_R, Y_{1,s}^\perp \rangle\, e^\frac{r^2}{4}\, d\calH^n - \int_{\Sigma_{0,s}} \langle D\varphi_R, Y_{0,s}^\perp \rangle \, e^\frac{r^2}{4}\, d\calH^n\, .
 \end{split}
 \end{equation}
 Again write the end $\Sigma_{1,R_0} \subset \Sigma_1$ as an exponential normal graph over $\bar{E}_{0,{R_0}}$  with height function $u : \bar{E}_{0,R} \ra \R$. We fix $x_0 \in \Sigma_0\setminus B_{R_0}(0)$ and choose a local parametrisation of $\Sigma_0$ in a neighbourhood of $x_0$, $F:\Omega \ra \Sigma_0$ such that $F^{-1}$ are normal coordinates around $x_0$, $F(0) = x_0$, such that the second fundamental form $h_{\Sigma_0}$ of $\Sigma_0$ is diagonalised at $x_0$ with principal curvatures $\lambda_1, \ldots,, \lambda_n$. Thus we obtain a local parametrisation of $\Sigma_1$ in a neighborhood of $y_0= x_0 + u(x_0) \nu_0(x_0)$, where $\nu_0$ is a choice of unit normal vectorfield of $\Sigma_0$ via
 $$ \tilde{F}(p)= F(p) + u(p) \nu_0(p)\, .$$
 Then again by the formulas in Appendix \ref{app:NG} we have that
 $$ \nu_1(y_0) = - \nabla_{\Sigma_0}u (x_0)+ \nu_0(x_0) + Q_0(x_0,u,\nabla u) $$
 where $|Q_0(p,u,Du)| \leq C(r^{-1}|u| + |\nabla u|^2)$. Note further that
 $$ |y_0| = (r(x_0)^2 + 2 \langle x_0, \nu_0(x_0)\rangle u(x_0) + u^2(x_0))^\frac{1}{2} = r(x_0)+ Q_1(x_0,u)\, , $$
 where $|Q_1(p,u)| \leq C r^{-2}|u|$. Thus
 $$ \left \langle \frac{y_0}{|y_0|}, \nu_1(y_0)\right \rangle =  \left \langle \frac{x_0}{|x_0|}, \nu_0(x_0)\right \rangle - \partial_r u(x_0) + r^{-1}(x_0) u(x_0) + Q_2(x_0,u,\nabla u) $$
 where $|Q_2(p,u,Du)| \leq C(r^{-2}|u| + r^{-1} |\nabla u|^2)$. 
 
 For the Jacobian of the map $x \mapsto x + u \nu_0$ we obtain from \eqref{eq:ng.2b}
 $$ \sqrt{\det g} = 1 + Q_3(x_0,u)\, ,$$
 where $|Q_3(p,u)| \leq C r^{-1}|u|$. Furthermore,
 $$ e^\frac{|y_0|^2}{4} = e^\frac{|x_0|^2}{4} + e^\frac{|x_0|^2}{4} Q_4(x_0,u)\, ,$$
 where $|Q_4(p,u)| \leq C r^{-1}|u|$ and
 $$ \varphi_R'(|y_0|) = \varphi_R'(|x_0|) + Q_5(x_0, u)$$
 where $|Q_5(p,u)| \leq C\delta^{-2} r^{-2}|u|$.
 
  We now denote $ v_1:= \langle Y_{1,s}, \nu_1\rangle$ and $v_0:= \langle Y_{0,s}, \nu_1\rangle$,
 which are Jacobi fields along $\Sigma_{1,s}$ and $\Sigma_{0,s}$ respectively. We write $\tilde{v}_1(x_0) = v_1(y_0)$. Note that by Lemma \ref{lem: second jacobi estimate} we have that
 \begin{equation}\label{eq:10b2}
  |\tilde{v}_1 - v_0| \leq C r^{-n-1+\varepsilon} e^{- \frac{r^2}{4}}\, ,
  \end{equation}
  for some $\varepsilon\in(0,1)$.
Note further that
$$ \langle \nu_0(x_0), Y_{0,s}(x_0)\rangle  = r(x_0) \psi(x_0) + Q_7(x_0),$$
where $|Q_7(x,u,\nabla u)| \leq C r^{-1} $. Combining all these estimates, we arrive at
 \begin{equation}\label{eq:10c}
 \begin{split}
\varphi_R'&(|y_0|) \left \langle \frac{y_0}{|y_0|}, Y^\perp_{1,s}(y_0)\right \rangle e^\frac{|y_0|^2}{4} \sqrt{\det g(y_0)} - \varphi_R'(|x_0|) \left \langle \frac{x_0}{|x_0|}, Y^\perp_{0,s}(x_0)\right \rangle e^\frac{|x_0|^2}{4}
\\
&= \varphi'_R(|x_0|)\left(r^{-1}u(x_0) -\partial_r u(x_0)\right)\langle \nu_0(x_0), Y_{0,s}(x_0)\rangle e^\frac{|x_0|^2}{4} + e^\frac{|x_0|^2}{4} Q_8(x,u,\nabla u)\\
&= \varphi'_R(|x_0|)\psi \left(u(x_0) - r\partial_r u(x_0)\right) e^\frac{|x_0|^2}{4} + e^\frac{|x_0|^2}{4} Q_9(x,u,\nabla u)
\end{split}
\end{equation}
where 
\begin{equation}
\begin{split}
&|Q_8(x,u,\nabla u)| \leq C \delta^{-2} (r^{-1} |u| + r^{-2} |\tilde{v}_1-v_0| + |\nabla u|^2),\\
& |Q_9(x,u,\nabla u)| \leq C \delta^{-2} (r^{-1} |u| + r^{-2} |\tilde{v}_1 -v_0| + r^{-1} |\partial_r u| + |\nabla u|^2),
\end{split}
\end{equation}
and both $Q_8$ and $Q_9$ are supported in $\bar{B}_R\setminus B_{R-\delta}$. This implies that
\begin{equation}\label{eq:10f}
 \begin{split}
\int_{\Sigma_{1,s}} &\langle D\varphi_R, Y_{1,s}^\perp \rangle\, e^\frac{r^2}{4}\, d\calH^n - \int_{\Sigma_{0,s}} \langle D\varphi_R, Y_{0,s}^\perp \rangle \, e^\frac{r^2}{4}\, d\calH^n \\
&= \int_{\Sigma_{0,s}} \varphi'_R \psi(\theta) \left(u - r\partial_r u \right) e^\frac{r^2}{4}\, d\calH^n + 
\int_{\Sigma_{0,s}} Q_9\,  e^\frac{r^2}{4}\, d\calH^n
\end{split}
\end{equation}
We write $u = r^{-n-1}\hat{u} e^{-\frac{r^2}{4}}$ and thus
\begin{equation}\label{eq:12}
 u - r \partial_r u =  - \partial_r\hat{u}\, r^{-n}e^{-\frac{r^2}{4}} + \hat{u} \left((n+2) r^{-1-n} + \frac{1}{2}r^{1-n}\right) e^{-\frac{r^2}{4}}\, .
\end{equation}
Combining \eqref{eq:10b} with \eqref{eq:10f}, \eqref{eq:12} we obtain
\begin{equation}\label{eq:13}
 \begin{split}
\frac{d}{ds} &\calE_{\Sigma_0(s),\Sigma_1(s),\varphi_R} = \frac{1}{2}\int_{\Sigma_{0,s}}r^{1-n}\varphi'_R(r)\psi(\theta) \hat{u} \, d\calH^n \\
&+ \int_{\Sigma_{0,s}}\varphi'_R\left((n+2) r^{-1-n}\hat{u} - r^{-n} \partial_r\hat{u}\right) +  Q_9\,  e^\frac{r^2}{4}\, d\calH^n\\
&=: \frac{1}{2}\int_{\Sigma_{0,s}}r^{1-n}\varphi'_R(r)\psi(\theta) \hat{u} \, d\calH^n + \int_{\Sigma_{0,s}} Q_{10}\,  e^\frac{r^2}{4}\, d\calH^n\, ,
\end{split}
 \end{equation}
 where we can estimate, using \eqref{eq:10b2} and \eqref{eq:12},
 $$ |Q_{10}| e^\frac{r^2}{4} \leq \delta^{-2} C \left( r^{-n-3+\varepsilon} + r^{-n-1} |\hat{u}| + r^{-n-2}|\partial_r \hat{u}| + r^{-2n-2} |\nabla \hat{u}|^2 e^{-\frac{r^2}{4}} \right) \, ,$$
 and $Q_{10}$ is supported in $B_R(0)\setminus B_{R-\delta}(0)$.
 The estimates in Theorem \ref{theo-ptwise-est} now directly imply that
 $$ \int_{\Sigma_{0,s}} Q_{10}\,  e^\frac{r^2}{4}\, d\calH^n \ra 0 $$
 as $R \ra \infty$. Since $\hat{u}$ is asymptotically homogenous of degree zero, it follows from Corollary \ref{coro-def-trace-inf} that as $R\ra \infty$
 \begin{equation}\label{eq:14}
 \begin{split} \frac{1}{2}\int_{\Sigma_{0,s}}r^{1-n}\varphi'_R(r)\psi(\theta) \hat{u} \, d\calH^n &\ra \frac{1}{2} \int_{-\delta}^0 \varphi'(r) dr \int_\Gamma \psi \hat{a} \, d\calH^{n-1}\\
 &\ \ \ \ = - \frac{1}{2} \int_\Gamma \psi \hat{a} \, d\calH^{n-1}\, .
\end{split}
 \end{equation}
 Note that by Proposition \ref{thm:1.3} the limit
 $$\calE_{\Sigma_{0,s},\Sigma_{1,s}, \varphi} = \lim_{R\ra \infty} \calE_{\Sigma_{0,s},\Sigma_{1,s}, \varphi_R}$$
 exists, and the convergence is uniform in $\delta$. Furthermore $\calE_{\Sigma_{0,s},\Sigma_{1,s}, \varphi}$
 is differentiable with 
 $$ \frac{d}{ds} \calE_{\Sigma_0(s),\Sigma_1(s),\varphi} = - \frac{1}{2} \int_\Gamma \psi \hat{a} \, d\calH^{n-1} \, .$$
This implies that also $\calE_{\Sigma_0(s),\Sigma_1(s)}$ is differentiable with
$$ \frac{d}{ds} \calE_{\Sigma_0(s),\Sigma_1(s)} = - \frac{1}{2} \int_\Gamma \psi \hat{a} \, d\calH^{n-1}\, .$$
\end{proof}
 
\section{Generic Uniqueness}\label{gen-uniqueness-section}

We can combine the previous computations with the results of \cite{Bernstein17} and \cite{BernsteinWang17} to prove the following generic uniqueness theorem, which follows the strategy of White \cite{White87}.

We let $\Pi:\mathcal{ACE}_n^{k,\alpha}(\Gamma)\rightarrow C^{k,\alpha}(\Gamma;\R^{n+1})$, $k\geq 2$, $\alpha\in(0,1)$,  be the boundary map from the space of asymptotically conical expanders denoted by $\mathcal{ACE}_n^{k,\alpha}(\Gamma)$, which assigns to each expander the link of its asymptotic cone: see  \cite[Section 2]{BernsteinWang17} for definitions. We note that by work of Bernstein-Wang \cite{BernsteinWang17}, this map is Fredholm of degree $0$. 

\begin{theorem} \label{gen-uni-0-rel-ent}The set of regular values of the boundary map $$\Pi:\mathcal{ACE}_n^{k,\alpha}(\Gamma)\rightarrow C^{k,\alpha}(\Gamma;\R^{n+1}), \quad k\geq 5,\quad \alpha\in(0,1),$$ with more than one preimage with vanishing relative entropy is of codimension 1. In particular, the set of $C^{k,\alpha}$ asymptotic cones which possess multiple corresponding expanders in $\mathcal{ACE}_n^{k,\alpha}(\Gamma)$ with vanishing relative entropy is of first category. Moreover, the set of $C^{k,\alpha}$ asymptotic cones $C(\Gamma)$ that are smoothed out by more than one locally entropy minimising expander in $\mathcal{ACE}_n^{k,\alpha}(\Gamma)$ is of first category in the Baire sense. 
 \end{theorem}
 \begin{remark}
 The reason why Theorem \ref{gen-uni-0-rel-ent} asks for some amount of regularity $k\geq 5$ of the link of the asymptotic cone is due to the estimates of the first part of Theorem \ref{theo-ptwise-est} in order to make sense of the trace at infinity $\tr_{\infty}^0(\hat{u})$ in Corollary \ref{coro-def-trace-inf}.
 \end{remark}
\begin{proof}
First, by  \cite[Theorem 1.1, equation (2)]{BernsteinWang17}, the projection map $$\Pi:\mathcal{ACE}_n^{k,\alpha}(\Gamma)\rightarrow C^{k,\alpha}(\Gamma;\R^{n+1}), \quad k\geq 4,\quad \alpha\in(0,1),$$ is a smooth map of Fredholm index $0$. Moreover, \cite[Corollary 1.2]{BernsteinWang17} asserts that the set of regular values of $\Pi$ is an open and dense subset of $C^{k,\alpha}(\Gamma;\R^{n+1})$. 

Now, consider the set of regular values of $\Pi$ with more than one preimage with vanishing relative entropy:
\begin{eqnarray*}
\mathcal{C}:=&\big\{(\Sigma_1,\Sigma_0):\quad \text{$\Sigma_i\in \mathcal{ACE}_n^{k,\alpha}(\Gamma)$, $i=0,1$, $\Pi(\Sigma_0)=\Pi(\Sigma_1)$,}\\
&\text{ regular value of $\Pi$, $\Sigma_0\neq\Sigma_1$, $\mathcal{E}_{\Sigma_0,\Sigma_1}=0$}\big\}.
\end{eqnarray*}
This set has been originally introduced by White in \cite[Section 7]{White87} in the context of minimal surfaces.

Our goal is to prove that $\mathcal{C}$ has codimension $1$.
Let $\Sigma_i$, $i=0,1$ be two expanders lying in $\mathcal{C}$ and let $C(\Gamma)$ be their common asymptotic cone. Let $F_i:U_i\rightarrow \mathcal{ACE}_n^{k,\alpha}(\Gamma)$, $i=0,1$ be corresponding charts where $U_i$, $i=0,1$ are neighborhoods of the link $\Gamma$ in $C^{k,\alpha}(\Gamma,\R^{n+1})$: these charts are provided by \cite[Theorem 7.1]{BernsteinWang17}. Let us consider the following relative entropy defined on $U_0\cap U_1$:
\begin{eqnarray*}
\tilde{\mathcal{E}}: \Gamma'\in U_0\cap U_1\rightarrow \mathcal{E}_{F_0(\Gamma'),F_1(\Gamma')}\in\R.
\end{eqnarray*}
Thanks to Theorem \ref{thm:1.5}, the functional $\tilde{\mathcal{E}}$ is differentiable at $\Gamma$ and its differential in direction $\psi$ is:
\begin{eqnarray}
d_{\Gamma}\tilde{\mathcal{E}}(\psi)=-\frac{1}{2} \int_\Gamma \psi \tr_{\infty}^0(\hat{u})\, d\calH^{n-1}\, ,\label{diff-fct-rel-ent-chart}
\end{eqnarray}
where $\hat{u}=r^{n+1}e^{\frac{r^2}{4}}(F_1(\Gamma)-F_0(\Gamma)).$
We invoke the results of Bernstein  \cite[Theorem]{Bernstein17} stating that if $\tr_{\infty}^0(\hat{u})$ vanishes, then the two asymptotically conical expanders $\Sigma_0$ and $\Sigma_1$ have to coincide. Thus (\ref{diff-fct-rel-ent-chart}) implies that the map $\Pi$ is a local submersion. The result then follows from the implicit function theorem applied to the functional $\tilde{\mathcal{E}}$, provided it is $C^1$:
\begin{claim}\label{claim-c0-trace}
The map $$\Gamma'\in U_0\cap U_1\rightarrow \tr_{\infty}^0\Big(r^{n+1}e^{\frac{r^2}{4}}\left(F_1(\Gamma')-F_0(\Gamma')\right)\Big)\in C^0(\Gamma,\R^{n+1}),$$ is continuous.
\end{claim}

\begin{proof}[Proof of Claim \ref{claim-c0-trace}]
Given a positive $\varepsilon$ and a radius $R$, there exists a neighborhood $U_R\subset U_0\cap U_1$ of $\Gamma$ such that if $\Gamma'\in U_R$:
\begin{equation*}
\sup_{\partial B(0,R)}r^{n+1}e^{\frac{r^2}{4}}\left|(F_1(\Gamma')-F_0(\Gamma'))-(F_1(\Gamma)-F_0(\Gamma))\right|\leq \frac{\varepsilon}{2},
\end{equation*}
by the continuity of the charts $F_i$, $i=0,1$.

Now, according to the first part of Theorem \ref{theo-ptwise-est}, there is a positive radius $R_0=R_0(\Gamma)$ sufficiently large and a positive constant $C=C(\Gamma)$ such that:
\begin{equation*}
\left|r^{n+1}e^{\frac{r^2}{4}}\left(F_1(\Gamma')-F_0(\Gamma')\right)-\tr_{\infty}^0\left(r^{n+1}e^{\frac{r^2}{4}}\left(F_1(\Gamma')-F_0(\Gamma')\right)\right)\right|\leq \frac{C}{r^2},
\end{equation*}
if $r\geq R_0$ and  $\Gamma'\in U_1\cap U_2$. Indeed, it suffices to integrate (\ref{evo-equ-u-hat}) along the radial vector field $\frac{x}{2}$ by using the pointwise bounds (\ref{ptwise-est-u}) as in the proof of Corollary \ref{coro-def-trace-inf}.

Claim \ref{claim-c0-trace} thus follows from an $\varepsilon/2$ argument.
\end{proof}

 Assume that we have two locally entropy minimising expanders $\Sigma_0, \Sigma_1$ asymptotic to $C(\Gamma)$. We can thus construct out of $\Sigma_1\cap B_R(0)$ a competitor to $\Sigma_0$ by adding a ribbon inside $\mathbb{S}_R(0)$. Note that by \eqref{ptwise-est-u} the area of the ribbon is bounded by $C R^{-2}e^{-\frac{R^2}{4}}$. We thus obtain that 
$$ \mathcal{E}_{\Sigma_0,\Sigma_1}(R) \leq C R^{-2} $$
By interchanging the roles of $\Sigma_0$ and $\Sigma_1$ we similarly see that
$$  \mathcal{E}_{\Sigma_0,\Sigma_1}(R) \geq - C R^{-2}\ . $$
This implies that $\mathcal{E}_{\Sigma_0,\Sigma_1} =0$. We can thus apply the result above to show that generically entropy minimising expanders are unique.
\end{proof}
\newpage
\begin{appendix}
\section{Geometry of normal graphs}\label{app:NG}
Let again $\Sigma_0,\Sigma_1$ be two expanders asymptotic to the cone $C(\Gamma)$, and assume $ \bar{E}_{1,R_0} = \text{graph}_{\bar{E}_{0,R_0}}(u)$ with $u : \bar{E}_{0,R_0} \ra \R$.  Let $p \in \Sigma_0$ and choose a local parametrisation $F$, parametrising an open neighbourhood $U$ of $p$ in $\Sigma_0$ such that $F(0)=p$. We can assume that $ g_{ij}=\langle \partial_i F, \partial_j F \rangle$ satisfies
$$ g_{ij}\big|_{x=0} = \delta_{ij}  \text{ and } \partial_kg_{ij}|_{x=0} = 0\, .$$
For simplicity we can furthermore assume that the second fundamental form $(h_{ij})$ is diagonalised at $p$ with eigenvalues $\lambda_1,\ldots, \lambda_n$. A direct calculation, see \cite[(2.27)]{Wang14}, yields that the normal vector $\nu_1(q)$, where $q = p + u(p) \nu_0(q)$, is co-linear to the vector
$$ N = - \sum_{i=1}^n \frac{\partial_i u}{1 - \lambda_i u} \partial_iF\bigg|_{x=0} + \nu_0(p)\ . $$
Denoting the shape operator $S = (h^i_{\ j})$ we see that thus in coordinate free notation
\begin{equation}\label{eq:ng.1}
\nu_1(q) = v^{-1} \left( - (\text{Id} - uS)^{-1}\nabla_{\Sigma_0} u + \nu_0\right) (p)\, ,
\end{equation}
where $v:=(1 + |(\text{Id} - uS)^{-1}(\nabla_{\Sigma_0} u)|^2)^{\frac{1}{2}}$.
This implies
\begin{equation}\label{eq:ng.2}
\langle q, \nu_1(q) \rangle = v^{-1} \left( u + \langle p, \nu_0(p) \rangle - \big\langle p, (\text{Id} - uS)^{-1}\nabla_{\Sigma_0} u \big\rangle \right)\, .
\end{equation}
For the induced metric $\tilde g$ one obtains in the above coordinates at $p$, again see \cite[(2.32)]{Wang14},
\begin{equation}\label{eq:ng.2b}
 \tilde g_{ij} = (1-\lambda_iu)(1-\lambda_ju) \delta_{ij} + \partial_i u \partial_j u
 \end{equation}
which implies
\begin{equation}\label{eq:ng.3}
\tilde g^{ij} = \frac{\delta^{ij}}{(1-\lambda_i u)(1-\lambda_j u)} -  v^{-2}\frac{\partial_i u}{(1-\lambda_i u)^2}\frac{\partial_j u}{(1-\lambda_j u)^2}\, .
\end{equation}
Furthermore, from \cite[(2.30)]{Wang14} we have
\begin{equation}\label{eq:ng.4}
\begin{split}
\tilde h_{ij} &= \langle \partial^2_{ij} \tilde F,  \nu_N \rangle \\ &= v^{-1} \Big( \frac{\lambda_i}{1-\lambda_i u} \partial_i u \partial_ju + \frac{\lambda_j}{1-\lambda_j u} \partial_i u \partial_ju\\
&\qquad\quad\ + \sum_k \frac{u}{1-\lambda_k u}\, \partial_k u \,\partial_i h_{jk} + h_{ij} - \lambda_i\lambda_j u\, \delta_{ij} + \partial^2_{ij} u\Big)\, .
\end{split}
\end{equation}
Since $H_{\Sigma_1}(p) = \tilde g^{ij}(p) \tilde h_{ij}(p)$ we see that
\begin{equation}\label{eq:ng.5}
\begin{split}
H_{\Sigma_1}(p) &= \Delta_{\Sigma_0} u + u |A_{\Sigma_0}|^2 + H_{\Sigma_0}(q) + Q(x,u,\nabla^{\Sigma_0} u , \nabla_{\Sigma_0}^2 u)\, ,
\end{split}
\end{equation}
where $Q(x,u,\nabla^{\Sigma_0} u , \nabla_{\Sigma_0}^2 u)$ is quadratic in $u, \nabla^{\Sigma_0} u, \nabla_{\Sigma_0}^2u$. 

\section{Interpolation inequalities} \label{app:interpolate}
We recall the following standard interpolation inequalities in multiplicative form. 
\begin{lemma}\label{lemm:interpolation}
Suppose that $u \in C^{k}(B_{2})$, then for $j< k$,
\[
\Vert D^{j} u \Vert_{C^{0}(B_{1})} \leq C \Vert u\Vert_{C^{0}(B_{2})}^{1-\frac{j}{k}} \Vert D^{k} u \Vert_{C^{0}(B_{2})}^{\frac j k}
\]
for $C=C(n,k)$. Similarly, if $u \in C^{k,\alpha}(B_{2})$, then for $j+\beta < k+\alpha$,
\[
[ D^{j} u ]_{\beta;B_{1}} \leq C \Vert u\Vert_{C^{0}(B_{2})}^{1-\frac{j+\beta}{k+\alpha}} [ D^{k} u ]_{\alpha ; B_{2}}^{\frac{j+\beta}{k+\alpha}}
\]
for $C=C(n,k,\alpha,\beta)$.
\end{lemma}
These follow in a similar manner to the linear inequalities given in \cite[Lemma 6.32]{giltru}, except in the proof one should optimize with respect to the parameter $\mu$ rather than just choosing $\mu$ sufficiently small. Alternatively, see \cite[Lemma A.2]{Hormander:geodesy}. 

\end{appendix}

\providecommand{\bysame}{\leavevmode\hbox to3em{\hrulefill}\thinspace}
\providecommand{\MR}{\relax\ifhmode\unskip\space\fi MR }
\providecommand{\MRhref}[2]{%
  \href{http://www.ams.org/mathscinet-getitem?mr=#1}{#2}
}
\providecommand{\href}[2]{#2}

\end{document}